\documentclass[11pt]{amsart}
\usepackage{amssymb}
\usepackage{amsmath}
\usepackage{amsthm}
\usepackage{graphicx}
\usepackage{macros}
\usepackage{bm}
\usepackage[usenames,dvipsnames]{xcolor}
\usepackage{subcaption}
\usepackage{tikz}
\usepackage{pgfplots} 
\usetikzlibrary{3d,calc}
\usepackage{mwe}

\usetikzlibrary{arrows.meta,
                positioning,
                quotes}

\usepackage[T1]{fontenc}
\usepackage[english]{babel}
\pgfplotsset{compat=newest}

\usepackage{float}
\usepackage{color}
\usepackage{hyperref}
\hypersetup{
	colorlinks,
	citecolor= Brown,
	linkcolor= link,
	urlcolor= link
}

\usepackage{appendix}
\usepackage{array}
\usepackage{footnote}
\usepackage{booktabs}
\usepackage{hhline}
\makesavenoteenv{tabular}

\usepackage{algpseudocode}
\usepackage{algorithm}

\usepackage{caption}

\usepackage{bbm}

\usepackage{pgfplots} 
\usepackage{textcomp}

\usepackage{scrextend}
\deffootnote{0em}{1.6em}{\enskip}
\usepackage{neuralnetwork}
\graphicspath{{./figures/}}

\usetikzlibrary{positioning}	
\tikzset{%
	every neuron/.style={
		circle,
		draw,
		minimum size=1cm
	},
	neuron missing/.style={
		draw=none, 
		scale=4,
		text height=0.333cm,
		execute at begin node=\color{black}$\vdots$
	},
}

\usepackage[top=2cm,bottom=2cm,left=2.5cm,right=2.5cm]{geometry}
\definecolor{link}{rgb}{0.18,0.25,0.63}
\definecolor{myred}{rgb}{0.7,0.25,0.2}
\definecolor{mygray}{rgb}{0.8,0.8,0.8}

\numberwithin{equation}{section}

\newcommand{\cmmnt}[1]{}

\makeatletter
\g@addto@macro{\endabstract}{\@setabstract}
\newcommand{\authorfootnotes}{\renewcommand\thefootnote{\@fnsymbol\c@footnote}}%
\makeatother

\definecolor{myred}{rgb}{0.78,0.20,0.00}

\newtheorem{definition}{Definition}[section]
\newtheorem{theorem}[definition]{Theorem}
\newtheorem{proposition}[definition]{Proposition}
\newtheorem{lemma}[definition]{Lemma}
\newtheorem{remark}[definition]{Remark}
\newtheorem{assumption}[definition]{Assumption}

\newcommand{\R}{ \mathbb{R}}

\begin{document}
%	\title{Optimality conditions for optimal control of learning-informed nonsmooth PDEs}
%	\author{Guozhi Dong \and Michael Hinterm\"uller \and Kostas Papafitsoros \and Kathrin V\"olkner
%	}
%%	\author[1,2]{Guozhi Dong}
%%	\author[1,2]{Michael Hinterm\"uller}
%%	\author[1]{Kostas Papafitsoros}
%%	\author[2]{Kathrin V\"olkner}
%%	\affil[1]{Weierstrass Institute for Applied Analysis and Stochastics, Mohrenstrasse 39, 10117 Berlin, Germany}
%%	\affil[2]{Institute for Mathematics, Humboldt University of Berlin, Unter den Linden 6, 10099 Berlin, Germany}
%
%	%\date{ }
%%	\setcounter{Maxaffil}{0}
%%	\renewcommand\Affilfont{\itshape\small}
%	\maketitle

 \begin{center}
 \Large
   \textbf{\textsc{First-order conditions for the optimal control of learning-informed nonsmooth PDEs}} \par \bigskip \bigskip
   \normalsize
    \textsc{Guozhi  Dong}\textsuperscript{$\,1$},
  \textsc{Michael Hinterm\"uller}\textsuperscript{$\,2,$$3$}, \textsc{Kostas Papafitsoros}\textsuperscript{$\,4$},  \textsc{Kathrin V\"olkner}\textsuperscript{$\,5,$$6$}
\let\thefootnote\relax\footnote{
\textsuperscript{$1$}School of Mathematics and Statistics, HNP-LAMA, Central South University, Lushan South Road 932, 410083 Changsha, China
}

\let\thefootnote\relax\footnote{
\textsuperscript{$2$}Institute for Mathematics, Humboldt-Universit\"at zu Berlin, Unter den Linden 6, 10099 Berlin, Germany}

\let\thefootnote\relax\footnote{
\textsuperscript{$3$}Weierstrass Institute for Applied Analysis and Stochastics (WIAS), Mohrenstrasse 39, 10117 Berlin, Germany}

\let\thefootnote\relax\footnote{
\textsuperscript{$4$}School of Mathematical Sciences, Queen Mary University of London, Mile End Road, E1 4NS, UK}
\let\thefootnote\relax\footnote{
\textsuperscript{$5$}Department of Mathematics, Freie Universit\"at Berlin, Arnimallee 6, 14195 Berlin, Germany}
\let\thefootnote\relax\footnote{
\textsuperscript{$6$}Zuse Institute Berlin, Takustrasse 7, 14195 Berlin, Germany}

\let\thefootnote\relax\footnote{
\vspace{-12pt}
\begin{tabbing}
\hspace{3.2pt}Emails: \= \href{mailto:guozhi.dong@csu.edu.cn}{\nolinkurl{guozhi.dong@csu.edu.cn}},
\href{mailto:Hintermueller@wias-berlin.de}{\nolinkurl{hintermueller@wias-berlin.de}},
 \href{mailto: k.papafitsoros@qmul.ac.uk}{\nolinkurl{k.papafitsoros@qmul.ac.uk}},\\
 \> \href{mailto: kathrin.voelkner@fu-berlin.de}{\nolinkurl{kathrin.voelkner@fu-berlin.de}}
  \end{tabbing}
}
\end{center}
\vspace{-0.8cm}

	\begin{abstract}
 In this paper we study the optimal control of a class of semilinear elliptic partial differential equations which have nonlinear constituents that are only accessible by data and are approximated by nonsmooth ReLU neural networks. The optimal control problem  is studied in detail. In particular, the existence and uniqueness of the state equation are shown, and continuity as well as directional differentiability properties of the corresponding control-to-state map are established. Based on approximation capabilities of the pertinent networks, we address fundamental questions regarding approximating properties of the learning-informed control-to-state map and the  solution of  the corresponding optimal control problem.  Finally, several stationarity conditions are derived based on different notions of generalized differentiability.
		\vskip .3cm
		%{\bf AMS subject classifications.} \ { }
		%\vskip .3cm	
	\end{abstract}

\section{Introduction}
\setcounter{footnote}{0} 
\subsection{Context and motivation}

%\vspace{1cm}

In this paper we study the following prototypical optimal control problem:
\begin{equation}\label{P_intro}\tag{$P_{f}$}
\begin{aligned}
&\text{minimize }\quad J(y,u):= \frac{1}{2} \|y-g\|_{L^{2}(\Omega)}^{2} + \frac{\alpha}{2} \|u\|_{L^{2}(\om)}^{2},\quad \text{ over } (y,u)\in H_{0}^{1}(\om) \times L^{2}(\om),\\
&\text{subject to }\left \{
\begin{aligned}
-\Delta y + f(\cdot, y)&=u, \;\; \text{ in }\Omega,\\
y&=0, \;\; \text{ on }\partial \Omega,
\end{aligned}
\right. \quad \text{ and }\quad  u\in \mathcal{C}_{ad}.
\end{aligned}
\end{equation}
Here $\Omega$ denotes an open, bounded Lipschitz domain in $\RR^{d}$, $d\ge 2$, with boundary $\partial\Omega$, $g\in L^{2}(\om)$ is a given desired state, and $\alpha>0$ is fixed. Further, $\mathcal{C}_{ad}$ is the admissible set of controls $u$, which is a nonempty, bounded, closed and convex subset of the Lebesgue space $L^{p}(\om)$ for some $p\ge 2$. The state is given by $y$ and is supposed to lie in the Sobolev space $H_0^1(\Omega)$; see, e.g., \cite{Adams} for the latter. Moreover, $f:\Omega\times \R \to \R$ is assumed to be a rather general nonlinear and possibly nonsmooth function. 
A particular motivating example for this work is given by $f$ represented (or approximated) by a so-called ReLU artificial neural network. In this case $f$ is replaced by $\mathcal{N}$, a neural network function that has the Rectified Linear Unit (ReLU)  $\sigma(t):=\max(t,0)$ as the underlying activation function; see Section \ref{sec:deep_learning} for more details and definitions. When such an $\mathcal{N}$ approximates $f$, the semilinear partial differential equation PDE in \eqref{P_intro} is replaced by the following (nonsmooth) \emph{learning-informed} PDE
\begin{equation}\label{intro:state_N}
\left \{
\begin{aligned}
-\Delta y + \mathcal{N}(\cdot, y)&=u, \;\; \text{ in }\Omega,\\
y&=0, \;\; \text{ on }\partial \Omega.
\end{aligned}
\right. 
\end{equation}
The concept of learning-informed PDEs was recently introduced in \cite{DonHinPap20}. It aims at representing an (explicitly) unknown physical law  $f$ by a neural network (map) $\mathcal{N}$ which is learned from given data.  For instance, assume that a data set
\[D:=\{(y_{i}, u_{i}):\; \text{$y_{i}$ (approximately) solves the original unknown PDE for $u_{i}$},\;i=1\ldots, n_{D}\},\]
of pre-specified controls and associated state responses is at our disposal. Such a set may, for example, arise from measurements or computations. Then this data set can be used towards evaluation instances of $f$ via $f(x_{j}, y_{i}(x_{j}))\simeq u_{i}(x_{j})+ \Delta y_{i} (x_{j})$, where $\{x_{j}\}_{j=1}^{\ell}$, $\ell\in\mathbb{N}$, is a finite collection of evaluation locations in the domain $\Omega$. Using these instances as a training set, a neural network $\mathcal{N}$ can be computed in the context of supervised learning and take the role of an approximation of the unknown $f$. Via solving the associated PDE \eqref{intro:state_N}, this results in a learning-informed control-to-state map $S_{\mathcal{N}}: u\mapsto y$. In view of this, such a set-up can also be regarded as an example in the currently rather limited field of operator learning, which is the task of learning a map between infinite dimensional spaces; here from the control to the state space. In particular, it provides a case that depicts how low dimensional learning could help to identify infinite dimensional control-to-state operators. We also note that in this set-up, the training of the network $\mathcal{N}$ is assumed to be done in an \emph{offline} phase, meaning that the minimization problem that constitutes the training of $\mathcal{N}$ and the optimal control problem \eqref{P_intro} are decoupled. 
 In \cite{DonHinPap20}, a smooth version of the above framework was considered in order to learn the physical law that governs the separation of a (initially homogeneous) material mixture into two pure states. This law can be related to (a derivative of) a double-well potential type function $f$ in a stationary Allen-Cahn equation.  
 %Traditionally, the choice for the explicit form of  this double-well  function  has been  a matter of modeling choice rather than data-driven. 
 In the same work, an ordinary differential equation (ODE) version of \eqref{P_intro} was  used to approximate the physics behind magnetic resonance imaging (MRI) \cite{DonHinPap19} with the goal of reconstructing quantitative maps of biophysical parameters in the context of quantitative MRI.
 
 The regularity of the neural network function $\mathcal{N}$ is determined by the regularity of the underlying activation function. In this vein, we note that the set of functions represented by ReLU neural networks coincides with the family of piecewise affine maps; see Section \ref{sec:deep_learning}. Since the main objective in \cite{DonHinPap20} was to introduce the concept of learning-informed PDEs, only smooth neural networks were considered, corresponding to some smooth, typically sigmoidal activation function, such as, e.g.,\ tansig or arctan. This eases the subsequent optimization-theoretic treatment of the optimal control problem like the derivation of first-order optimality conditions. However, when it comes to deep learning applications nowadays, the nonsmooth ReLU is the most popular choice among activation functions. Its main advantages lie in its sparsification  properties, its easy computation and further advantages during training, such as for instance with regards to remedy the issue of vanishing gradients \cite{bengio2013representation, glorot2011deep}. The study of approximation capabilities of ReLU neural networks is another active field of research \cite{relu_Wsp, petersen2018optimal}. All these aspects are also relevant to our problem, in particular with respect to the capability of the learning-informed optimal control problem to approximate the original one. %All these,  motivate us further to study the  problem \eqref{P_intro} also in relation with its ReLU learning-informed version where $f$ is replaced by its approximation $\mathcal{N}$. 
 We mention already here that we consider the nonsmooth $\mathcal{N}$ to be monotonically increasing in the variable $y$. This guarantees uniqueness of solutions to the learning-informed state equation, resulting in a well-defined control-to-state map. Conceptually, this assumes that the  ground truth map $f$ is also monotonically increasing and the network has been trained sufficiently well, thus preserving this monotonicity. However, we point out that for homogeneous Dirichlet boundary conditions (like the one in \eqref{intro:state_N}) one can still show uniqueness of the solution of the state equation if the negative part of $\partial_{y} \mathcal{N}$ is sufficiently small (which is often the case for good enough approximations of monotone functions).  In this case it can still be shown that the corresponding PDE operator is strongly monotone and the Browder-Minty theorem can be applied.  Let us also point out that there are ways to enforce monotonicity during training  \cite{Daniels_monotone, nips_monotonic, sivaraman2020counterexample} via a sufficiently good derivative approximation (Sobolev training)  \cite{sobolev_training}.
 
%\vspace{1cm}
%-ReLU networks why is it relevant, connection to nonsmooth optimal control papers

Since $f$ and its approximation $\mathcal{N}$ are nonsmooth maps, one does not expect the associated control-to-state map $S$ to be Frech\'et differentiable. This  poses difficulties in the derivation of first-order optimality respectively stationarity conditions for the optimal control problem as well as its numerical treatment. Therefore, one needs to resort to more general notions of differentiability (than Fr\'echet differentiability); we refer to the introduction of \cite{christof} for a short review. In the same paper, the optimal control of a semilinear PDE with the special choice $f(\cdot, y)=\max(0,y)$ and in the absence of control constraints was studied, characterizing the Bouligand subdifferential of the control-to-state map and deriving so-called strong stationarity conditions. We remark that in our set-up this would correspond to a very simple and shallow ReLU network $\mathcal{N}$. Here, following \cite{christof} and also adapting certain results  from a recent preprint \cite{Betz},
%\footnote{\textsuperscript{$1$}\textcolor{red}{Perhaps this needs some change!}\textcolor{blue}{ We  accessed this preprint on March 1, 2021, via the author's webpage  (provided in our bibliography) which contained a link to the preprint in www.researchgate.net. However, we are aware that at the time of the submission of our paper the latter link is broken.}},
 we study in depth a  version of \eqref{P_intro} with a more general $f$ that still covers the ReLU neural network case. This function $f$ is assumed to be measurable in the first variable and locally Lipschitz and directionally differentiable in the second one.  For this general problem, we establish a series of stationarity respectively necessary first-order conditions. We begin by deriving purely primal optimality conditions ($B$-stationarity) which are equivalent to  the directional derivative of the reduced objective in feasible directions being non-negative, as well as conditions that are derived as a limit of a regularization scheme  \cite{barbu1984optimal, christof, mignot_puel} (weak stationarity). The latter conditions are refined in the case that $f$ is a piecewise differentiable function yielding $C$-stationarity.  Under additional constraint qualifications, we further improve $C$-stationarity  by establishing certain sign conditions on the dual variables. This leads to strong stationarity which is shown to be equivalent to $B$-stationarity in specific cases.

\subsection{Structure of the paper} 
In Section \ref{sec:deep_learning} we start by reviewing basic properties of ReLU neural networks.
%We focus on their structure as functions and in particular how this changes after smoothing them by smoothing the ReLU activation function (canonical smoothing). We are particularly interested in how this kind of smoothing preserves or not the monotonicity of the network. We also discuss known $W^{1,\infty}$ approximation properties of ReLU neural networks. 
In Section \ref{sec:ReLU_PDE}, we study the state equation of the optimal control problem \eqref{P_intro} with a rather general nonsmooth, nonlinear function $f$. Existence and uniqueness of solutions to this PDE are discussed, and continuity as well as directional differentiability properties of the corresponding control-to-state map are shown. We then study approximation results of this PDE on two levels: (i) $f$ is approximated by a ReLU neural network $\mathcal{N}$; (ii) $f$, or a network $\mathcal{N}$, is approximated by  some smoothed version of it.
% Here we provide a warning that such an approach should be used only as a tool to derive limiting stationary conditions and not  in order to use classical smooth numerical solvers, as the new "smoothed" learning-informed state equation might have multiple solution due to the potential loss of monotonicity of the network. 
 The corresponding general optimal control problem is studied in Section \ref{sec:wellposedness}. Existence and  approximation results stemming from the analogous results of the state equation are shown. 
 In Section \ref{sec:stationary}, we focus on different stationarity systems that are necessary for local minimizers of the optimal control problem, namely $B$-, weak, $C$- and strong stationarity. \\
%Section \ref{sec:algorithm}, introduces and analyzes a descent algorithm that deals directly with the nonsmooth optimal control problem. It is applied in Section \ref{sec:numerics}, to an optimal control problem with of a ReLU learning -informed semilinear PDE. Finally,  we conclude our work in Section \ref{sec:conclusion}.

We finally mention that in the companion paper \cite{DonHinPap22b}, by employing a descent algorithm inspired by the bundle-free method in \cite{Hint_Suro}, a numerical algorithm that treats the nonsmooth optimal control problem \eqref{P_intro} directly is proposed, analyzed and tested.

%\vspace{1cm}
%Selling our algorithm, mention discussion of canonical smoothing, what things can go wrong, non uniqueness of solutions. 

%\vspace{1cm}
%-monotone neural networks literature

%\subsection{Our contribution} 
%In this paper, we are developing the theoretical as well as the computational framework concerning optimal control of learning-informed nonsmooth PDEs.
%\textcolor{blue}{In the theretical side, we prove necessary conditions for the local optimality, in particular, we establish various stationarity conditions, including weak stationarity, B-stationarity, C-stationarity and strong stationarity. Then we specify connections and relations between each of these conditions.}
%We also notice some interesting properties regarding the fact that the nonsmoothness results from ReLU network structures which are compositions and linear combinations of `max' functions. This contributes not only to the theoretical analysis, but also helps in the numerical algorithm design and numerical analysis.
%In the numerical side, the proposed algorithm for optimal control of hybrid semilinear PDEs with ReLU network functions makes use of the auxiliary optimization problem as in \cite{Hint_Suro}. A particular feather is that the algorithm identifies a descent direction via smoothed `max' function approximating the derivatives of ReLU network functions. It is also shown that a B-stationarity (KKT condition) is possible to achieve using the proposed algorithm.

\section{ReLU neural networks}\label{sec:deep_learning}

%\subsection{Deep learning}\com{TODO}

In this work, we rely on ReLU neural networks as efficient approximators of an unknown and potentially complex function $f: \RR^{n_{0}}\to \RR^{n_{L}}$, with $n_0, n_L\in\mathbb{N}$, for which only a few (approximate) data values are available, namely $\{f_{j}\approx f(x_{j})\}_{j=1}^{n_{D}}$ where $n_D\in\mathbb{N}$ and $\{x_j\}_{j=1}^{n_D}\subset\Omega$. Given this setting, we assume that a network function $\mathcal{N}_{\theta}$ with parameters $\theta$ -- see Section \ref{sec:2.1} below for definitions -- has been trained to ``learn" the function $f$ in the context of supervised learning. This training is typically achieved by (approximately) solving 
\begin{equation}\label{supervised_learning}
\min_{\theta\in \RR^{p}}  \sum_{j=1}^{n_{D}} \ell(\mathcal{N}_{\theta}(x_{j}), f_{j}) + \mathcal{R}(\theta).
\end{equation}
Here, $\ell$ is a suitable loss function and $\mathcal{R}$ is an optional regularization term inducing some a priori properties on $\theta$. Clearly, the study of \eqref{supervised_learning} is an important research area in the field of deep learning \cite{goodfellow2016deep}, but here we rather assume that this learning process has been successfully completed such that we are equipped with a ReLU network $\mathcal{N}$ which approximates $f$ sufficiently well. As a consequence, in what follows we merely focus on the properties of such networks.

\subsection{Definition and basic properties}\label{sec:2.1}
Let $\mathbb{N}:=\{1,2, \ldots \}$.

\begin{definition}[Standard feedforward multilayer neural network]\label{def:neural_networks}
	Let $L\in \mathbb{N}$, network parameters $\theta=\left ((W_{1}, b_{1}), \ldots, (W_{L},b_{L} ) \right )$ with $W_{i}\in \mathbb{R}^{n_{i}\times n_{i-1}}$, $b_{i}\in \mathbb{R}^{n_{i}}$, for $i=1,\ldots, L$ and $n_{i}\in \mathbb{N}$ for $i=0,\ldots, L$. Furthermore let $\sigma:\mathbb{R}\to \mathbb{R}$ be an arbitrary function. We say that a function $\mathcal{N}: \mathbb{R}^{n_{0}}\to \mathbb{R}^{n_{L}}$ is a neural network with weight matrices $(W_{i})_{i=1}^{L}$, bias vectors $(b_{i})_{i=1}^{L}$ (the network parameters) and activation function $\sigma$ if $\mathcal{N}(x)$ can be defined through the following recursive relation for any $x\in \mathbb{R}^{n_{0}}$:
	\begin{align}
	z_{0}&= x,\label{defNN1}\\
	z_{\ell}&=\sigma \left (W_{\ell} z_{\ell-1} + b_{\ell} \right ), \quad \ell=1,\ldots,  L-1, \label{defNN2}\\
	\mathcal{N}(x)&= W_{L} z_{L-1} + b_{L}. \label{defNN3}
	\end{align}
	The action of the activation function $\sigma$ in \eqref{defNN2} is understood in a component-wise sense, i.e., for a vector $y=(y^{1}, \ldots, y^{n})\in \mathbb{R}^{n}$ we have $\sigma(y):=(\sigma(y^{1}), \ldots, \sigma(y^{n}))$. More compactly,  
 $\mathcal{N}$ reads
 	\begin{equation}\label{defNN1_comp} 
	\mathcal{N}(x)= T_{L} \circ \sigma(T_{L-1}) \circ \cdots \circ \sigma(T_{2}) \circ \sigma (T_{1})(x), \quad x\in \RR^{n_{0}},
	\end{equation}
	where, for every $\ell=1,\ldots, L$, $T_{L}$ denotes the affine transformation $\mathbb{R}^{n_{\ell-1}}\ni z\mapsto W_{\ell} z + b_{\ell}\in\mathbb{R}^{n_\ell}$.\\[0.5em]
	We say that $\mathcal{N}$ is a ReLU neural network if $\sigma$ is the ReLU (Rectified Linear Unit) activation function
	\begin{equation}
	\sigma(t)=\max(t,0), \quad t\in \mathbb{R}.
	\end{equation}
	
\end{definition}

Following the standard neural network terminology, we say that a neural network defined as in \eqref{defNN1}--\eqref{defNN3}, has \emph{$L$ layers} and \emph{$L-1$ hidden layers}, with the latter denoting the operations in \eqref{defNN2}. The final operation \eqref{defNN3} is called the \emph{output layer}. Furthermore, $n_{i}$ is called the number of \emph{neurons} in the $i$-th layer, $i=1, \ldots L$, which is the number of rows of the weight matrix $W_{i}$. The number of neurons of a given layer is also called the \emph{width} of that very layer, while the number of layers is called the \emph{depth} of the network.

We note that a neural network as a function does not necessarily admit a unique representation with respect to the weight matrices, the bias vectors and the activation functions.  Furthermore in Definition \ref{def:neural_networks}, the input of the $\ell$-th layer consists only of the output $z_{\ell-1}$ of the previous layer. A more general neural network definition would allow the input for each layer to depend on the output of all previous layers. In that case, every $W_{\ell}$ would be a weight matrix of size $\mathbb{R}^{n_{i}\times (\sum_{k=0}^{\ell-1} n_{k})}$. However, since every network of that type can be realized by a network as in Definition \ref{def:neural_networks} (see also \cite{relu_Wsp}) we will always relate to the more classical definition given above.

%\vspace{1em}
%\mycom{TODO: write why ReLU is a popular choice in practice (e.g. they mention sth in \cite{relu_Wsp} but to check in other places too)}
%\vspace{1em}

We are interested in the regularity and structural properties of the class of functions that are realized by ReLU neural networks. It turns out that the latter coincides with the class of \emph{continuous piecewise affine functions}, which we define next.

\begin{definition}[Continuous piecewise affine functions]\label{def:cpwl}
	Let $n_{0}\in \mathbb{N}$. We say that a function $\mathcal{F}: \mathbb{R}^{n_{0}}\to \mathbb{R}$ is continuous piecewise affine  if 
	%one of the following equivalent conditions hold:
	the following condition holds:
	\begin{itemize}
		\item $\mathcal{F}$ is continuous and there exist finitely many affine maps $f_{1}, \ldots, f_{p}: \mathbb{R}^{n_{0}}\to \mathbb{R}$ for some $p\in\mathbb{N}$ such that for every $x\in \mathbb{R}^{n_{0}}$, there exists an $i\in \{1,\ldots, p\}$ such that $\mathcal{F}(x)=f_{i}(x)$.
%		\item There exist finitely many pairs of subsets of $\mathbb{R}^{n_{0}}$ and distinct affine maps $((S_{1}, f_{1}), \ldots, (S_{q}, f_{q}))$, for some $q\in \mathbb{N}$, with $\emptyset \ne S_{i}\subset \mathbb{R}^{n_{0}}$, $\overline{\mathrm{int}(S_{i})}= S_{i}$, for every $i\in \{ 1, \ldots, q\}$, $\mathrm{int}(S_{i})\cap \mathrm{int}(S_{j})=\emptyset$, for every $i\ne j$, $\bigcup_{i=1}^{q} S_{i}=\mathbb{R}^{n_{0}}$ such that  if $x\in S_{i}$ then $\mathcal{F}(x)=f_{i}(x)$. Necessarily the boundary of each $S_{i}$ is a subset of a union of hyperplanes (that is, each $S_{i}$ is a polytope):
%		\[\partial S_{i}=\bigcup_{\substack{j\in \{ 1, \ldots , q\} \\ j\ne i \\ S_{i}\cap S_{j}\ne \emptyset }} S_{i}\cap S_{j}\subseteq \bigcup_{\substack{j\in \{ 1, \ldots , q\} \\ j\ne i \\ S_{i}\cap S_{j}\ne \emptyset }} \{x\in \mathbb{R}^{n_{0}}:\; f_{i}(x)=f_{j}(x)\}\]
%		\item There exist  finitely many affine maps $f_{1}, \ldots, f_{s}: \mathbb{R}^{n_{0}}\to \mathbb{R}$ for some $s\in\mathbb{N}$ and subsets $P_{1}, \ldots P_{k}\subset \{1, \ldots, s\}$, where each $P_{i}$ is of cardinality $n_{0}+1$ such that
%		\[\mathcal{F}=\sum_{j=1}^{k} s_{j}\left (\max_{i\in P_{j}} f_{i} \right ),\]
%		and $s_{j}\in \{ -1, 1\} $ for all $j=1,\ldots, k$.
	\end{itemize}
\end{definition}

Note that such a map $\mathcal{F}$ is not necessarily Fr\'echet differentiable, and see \cite{aliprantis_harris_tourky_2006, arora2018understanding,hinging_hyperplanes} for
characterizatons and properties of continuous piecewise affine functions.
\begin{theorem}
	A  function $\mathcal{N}:\mathbb{R}^{n_{0}}\to \mathbb{R}$ is a ReLU neural network if and only if it is a continuous piecewise affine function.
\end{theorem}

From the definition \eqref{defNN1}--\eqref{defNN3} it is clear that $\mathcal{N}:\mathbb{R}^{n_{0}} \to \mathbb{R}^{n_{L}}$, $n_{L}\ge 1$, is a $\mathbb{R}^{n_{L}}$-valued ReLU neural network if and only if $\mathcal{N}=(\mathcal{N}_{1}, \ldots, \mathcal{N}_{L})$ with each $\mathcal{N}_{i}: \mathbb{R}^{n_{0}}\to \mathbb{R}$, $i=1, \ldots, L$, is a scalar-valued  ReLU neural network. Thus, $\mathcal{N}$ is an $\mathbb{R}^{n_{L}}$-valued ReLU neural network if and only if it is an $\mathbb{R}^{n_{L}}$-valued continuous piecewise affine function, with the latter defined exactly as in Definition \ref{def:cpwl} with the only difference that the affine maps $f_{i}$ are $\mathbb{R}^{n_{L}}$-valued.

%In order to give an example, for $p\ge 2$ and $t_{1} \le \cdots \le t_{p-1}$, consider the following one dimensional CPWL function $\mathcal{F}$ with 
%\begin{equation}\label{cpwl_1d}
%\mathcal{F}(t) =
%\begin{cases}
%a_{1} t + \gamma_{1} &\text{ if } t\le t_{1}, \\
%a_{i} t + \gamma_{i} &\text{ if } t_{i-1} \le t \le  t_{i}, \quad  i=2, \ldots p-1,\\
%a_{p} t + \gamma_{p} & \text{ if }  t \ge t_{p-1}.
%\end{cases}
%\end{equation}
%Note that we assume that $(a_{i},\gamma_{i})_{i=1}^{p}$ satisfy the appropriate conditions such that $\mathcal{F}$ is continuous.
%Then it can be checked, see for instance \cite[Corollary 3.5]{aliprantis_harris_tourky_2006} that  $\mathcal{F}$ can be written as 
%\begin{align}\label{cpwl_1d_relu}
%\mathcal{F}(t)
%&=a_{1}t + \gamma_{1} + \sum_{i=1}^{p-1} (a_{i+1}-a_{i}) \mathrm{max}(t-t_{i}, 0).\nonumber\\
%&=a_{1}(\mathrm{max}(t,0)  -a_{1}\mathrm{max}(-t,0) +  \sum_{i=1}^{p-1} (a_{i+1}-a_{i}) \mathrm{max}(t-t_{i}, 0) + \gamma_{1}
%\end{align}
%This means that $\mathcal{F}$ can be realized as a ReLU neural network with one hidden layer having $p+1$ neurons. In particular,  $\mathcal{F}= T_{2}\circ \sigma (T_{1})$, where $T_{1}(t)= W_{1}t + b_{1}$, $T_{2}(z)=W_{2}z +b_{2}$ with $W_{1}=(1,-1, 1, 1, \ldots, 1)^{T}\in \mathbb{R}^{(p+1)\times 1}$, $b_{1}=(0,0,-t_{1}\ldots, -t_{p-1})^{T}\in \mathbb{R}^{(p+1)\times 1}$, and $W_{2}=(a_{1}, -a_{1},a_{2}-a_{1}, \ldots, a_{p}-a_{p-1})\in \mathbb{R}^{1\times (p+1)}$, $b_{2}=\gamma_{1}\in \mathbb{R}$.

\vspace{1em}\noindent

\subsection{Approximation results}

Next we recall several basic properties regarding the approximation capabilities of ReLU neural networks in the spirit of the several versions of the \emph{universal approximation theorem}; see \cite{Pin99} for a thorough review. In fact, in the case where the activation function is $k$-times differentiable and not a polynomial, the set of the corresponding one-hidden layer networks is dense in the set of $k$-times continuously differentiable functions in the topology of uniform convergence (for function values and derivatives up to order $k$) on compact sets. In the case of the ReLU activation function, the corresponding approximation results are less standard and are stated here for the ease of access.

Indeed, there exists a growing amount of literature that studies properties of ReLU networks with regards to their approximation of  functions of a given Sobolev regularity. Typically, special emphasis is given to the upper bound on the number of layers, neurons and nonzero weights needed to achieve approximation of a certain accuracy with respect to a Sobolev norm; see for instance \cite{relu_chainrule, relu_Wsp, yarotsky18a}. For our purposes, the following result taken from the aforementioned references will be of interest: Given an open, bounded domain $U\subset \mathbb{R}^{n_{0}}$ with Lipschitz boundary we have that
\begin{equation}\label{sobolev_approx_box}
\begin{aligned}
\text{for every $\epsilon>0$ and $f\in W^{1,\infty}(U)$}&\text{ there exists a ReLU network $\mathcal{N}_{\epsilon}:\mathbb{R}^{n_{0}}\to \mathbb{R}$ such that}\\
& \|\mathcal{N}_{\epsilon}-f\|_{W^{1,\infty}(U)}<\epsilon.
\end{aligned}
\end{equation}
Here, the Sobolev space $W^{1,\infty}(U)$ relates to Lipschitz functions on $U$; see \cite{Adams}. We also note that the density result in \eqref{sobolev_approx_box} is typically stated for $U=(-M,M)^{n_{0}}$ for a given $M>0$, but it is easy to see that it holds for every open, bounded $U\subset \mathbb{R}^{n_{0}}$ with Lipschitz boundary. Indeed, given such a domain $U$, consider a bounded, linear extension operator $T:W^{1,\infty}(U)\to W^{1,\infty}(\mathbb{R}^{n_{0}})$ such that $Tg$ has  compact support, which is common for every  $g\in W^{1,\infty}(U)$. Then we readily get
\begin{equation*}
\|\mathcal{N}_{\epsilon}-f\|_{W^{1,\infty}(U)} \le \|\mathcal{N}_{\epsilon} -Tf\|_{W^{1,\infty}((-M,M)^{n_{0}})}
\end{equation*}
for some large enough $M>0$.

\section{Nonsmooth and ReLU neural network informed PDE }\label{sec:ReLU_PDE}

We recall once again the structure of the learning-informed optimal control problem that motivates us in the following:
\begin{equation}\label{P}\tag{$P_{\mathcal{N}}$}
\begin{aligned}
&\text{minimize }\quad J(y,u):= \frac{1}{2} \|y-g\|_{L^{2}(\Omega)}^{2} + \frac{\alpha}{2} \|u\|_{L^{2}(\om)}^{2},\quad \text{ over } (y,u)\in H_{0}^{1}(\om) \times L^{2}(\om),\\
&\text{subject to }\left \{
\begin{aligned}
-\Delta y + \mathcal{N}(\cdot, y)&=u, \;\; \text{ in }\Omega,\\
y&=0, \;\; \text{ on }\partial \Omega,
\end{aligned}
\right. \quad \text{ and }\quad  u\in \mathcal{C}_{ad}.
\end{aligned}
\end{equation}
%Here $\om$ denotes an open, bounded, Lipschitz domain in $\RR^{d}$, with $d\ge 2$. 
Note that the governing state equation in \eqref{P} is a semilinear elliptic PDE with the neural network $\mathcal{N}$ (approximating some unknown $f$) as a constituent. 
%\textcolor{red}{Seems we do not have to repeat the control problem here:}
%\textcolor{blue}{
%We recall once again that the learning-informed PDE \eqref{P} motivates partially our investigation:
%\begin{equation}\label{P}
%\begin{aligned}
%-\Delta y + \mathcal{N}(\cdot, y)&=u, \;\; \text{ in }\Omega\\
%y&=0, \;\; \text{ on }\partial \Omega
%\end{aligned}
%\end{equation}
%Here $\om$ denotes an open, bounded, Lipschitz domain in $\RR^{d}$, with $d\ge 2$. The  $\mathcal{N}$  is a neural network function which is a constituent of the semilinear PDE. 
%%}
Concerning $\mathcal{N}:\RR^{d} \times \RR\to \RR$ we assume that it is a ReLU neural network which is monotonically increasing in the second variable. 
%Furthermore, recall again that $g\in L^{2}(\om)$ denotes the desired state, $\alpha>0$ is a regularization parameter, while $\mathcal{C}_{ad}$, is a nonempty, closed and convex subset of $ L^{p}(\om)$, for some $p\ge 1$. 
The objective $J$ is of tracking type and may be replaced by other suitable candidates allowing to prove existence of a solution to the associated optimal control problem. Such tracking objectives often arise, e.g., in engineering applications where one wishes to find a system state $y$ which is sufficiently close to some given desired state $g$ and which arises through some control action $u$ in the governing equation. The associated control is constrained by $\mathcal{C}_{ad}$, and its pertinent average $L^2$-cost is $\alpha>0$.  

As outlined above, our underlying assumption is that $\mathcal{N}$ is an approximation of an unknown, potentially nonsmooth function $f:\om\times \RR\to \RR$, with \eqref{P_intro} the corresponding optimal control problem. Thus, for most of our analysis we will adopt a more general perspective by considering functions $f$ that belong to a larger family than the one defined by ReLU neural networks only. In such a setting $f=\mathcal{N}$ becomes one particular instance. %However we will come back to the more specific problem \eqref{P} when for instance we will deal with canonical smoothings of $\mathcal{N}$.
Concerning such a general $f:\om\times \mathbb{R}\to \mathbb{R}$, throughout we invoke the following: 
%We note already that some conditions imply others \com{(check if this is the case)}.
\begin{align}
&\text{For every $y\in \mathbb{R}$ the function $x\mapsto f(x,y)$ is Lebesgue measurable. Moreover} \tag{$A_{1}$}\label{assumption_1}\\[0.6em]
&\hspace{5cm} f(\cdot, 0)\in L^{\infty}(\om). \nonumber\\[0.6em]
&\text{$f$ is  Lipschitz continuous in $y$ on bounded sets, that is, for every $M>0$, there exists a}\tag{$A_{2}$}\label{assumption_2}\\
&\text{ constant $L=L(M)>0$, such that for all $y_{1}, y_{2}\in (-M, M):$}\nonumber\\[0.6em] 
&\hspace{3cm} |f(x,y_{1})-f(x,y_{2})|\le L|y_{1}-y_{2}|, \quad \text{ for almost every $x\in\om$}.\nonumber\\[0.6em]
&\text{For almost every $x\in \om$ the function $y\mapsto f(x,y)$ is monotone increasing.} \tag{$A_{3}$}\label{assumption_3}
\end{align}
\begin{align}
&\text{For almost every $x\in \om$ the function $y\mapsto f(x,y)$ is directionally differentiable with}\tag{$A_{4}$}\label{assumption_4}\\
&\text{its directional derivative at $y\in \RR$ in direction $h\in\RR$ given by}\nonumber\\[0.6em]
&\hspace{3.5cm}    f_{x}'(y;h)=\lim_{t_{n}\to 0^{+}} \frac{f(x,y+t_{n} h)-f(x,y)}{t_{n}}. \nonumber%\\[0.6em]
\end{align}
Note that we do not necessarily require $y\mapsto f(x,y)$ to be strictly monotonically increasing. It is further clear that a ReLU neural network $\mathcal{N}:\RR^{d}\times \RR\to \RR$, restricted to $\om\times \RR$, which is monotonically increasing in the second variable satisfies assumptions \eqref{assumption_1}--\eqref{assumption_4}. In fact, it is even globally Lipschitz continuous in the second variable. 
%\com{(Write down an explicit formula for $\mathcal{N}_{x}'(y;h)$)}

Observe further that $\mathcal{N}$ is even Hadamard directionally differentiable with respect to the second variable. Then, using the chain rule for Hadamard directionally differentiable functions \cite[Proposition 2.47]{bonnans2013perturbation},
%({\color{red}Provide reference!}) 
we can state a recursion formula for $\mathcal{N}_{x}'(y;h)$. For this, we first confine ourselves to a two-layer architecture and generalize afterwards. Indeed, let $z:=(x,y)$ and $N^{(2)}(z)=W_{2}\cdot \sigma(W_{1} z +b_{1}) + b_{2}$, $W_{2}\in \RR^{1\times n_{1}}$, $W_{1}\in \RR^{n_{1}\times (d+1)}$, $b_{1}\in \RR^{n_{1}}$, $b_{2}\in\RR$. Then we have for any $y,h\in \RR$ that
\begin{align}\label{direct_deriv_N2}
(N^{(2)})_{x}'(y;h)=W_{2}\cdot \left ( \mathbbm{1}_{(0,\infty)}(W_{1}z+b_{1}) W_{1}(:,n_{0})h 
+  \mathbbm{1}_{\{0\}}(W_{1}z+b_{1})  \mathrm{max}(0,W_{1}(:,n_{0})h) \right ). 
\end{align}
Here, $ W_{1}(:,n_{0})$ denotes the last column of $W_{1}$, and  $\mathbbm{1}_{(0,\infty)}(W_{1}z+b_{1})$ is a diagonal matrix, whose diagonal consists of the vector resulting from the componentwise action  of the characteristic function $\mathbbm{1}_{(0,\infty)}:\mathbb{R}\to\{0,1\}$, with $\mathbbm{1}_{(0,\infty)}(t)=1$ if $t\in \mathcal{M}:=
(0,\infty)$ and $\mathbbm{1}_{(0,\infty)}(t)=0$ otherwise, on the vector $W_{1}z+b_{1}$ -- similarly for the second summand in \eqref{direct_deriv_N2}. In general, recursively for $N^{(\ell)}=W_{\ell}\sigma(N^{\ell-1}(z)) + b_{\ell}$ we have
\begin{align}\label{direct_deriv_N_ell}
(N^{(\ell)})_{x}'(y;h)=W_{\ell} \cdot \left ( \mathbbm{1}_{(0,\infty)}(N^{(\ell-1)}(z)) (N^{(\ell-1)})_{x}'(y;h) 
+ \mathbbm{1}_{\{0\}} (N^{(\ell-1)}(z)) \max(0,(N^{(\ell-1)})_{x}'(y;h))  \right ).
\end{align}
Comparing \eqref{direct_deriv_N2}--\eqref{direct_deriv_N_ell} with
	\begin{equation}\label{relu_prime}
	\nabla \mathcal{N}(x)=W_{L}  \cdot \sigma'( N^{(L-1)}(x)) \cdot W_{L-1} \cdot \ldots \cdot \sigma'(N^{(1)}(x))\cdot W_{1},
	\end{equation}
the weak gradient of $\mathcal{N}$ \cite{relu_chainrule}, we note that while \eqref{relu_prime} holds almost everywhere, the formulas for the directional derivatives hold at every point.

For our analysis we will make use of the space 
\[Y:=\{y\in H_{0}^{1}(\om):\; \Delta y \in L^{2}(\om) \},\]
which is a separable Hilbert space equipped with the inner product $(y,v)_{Y}:= \int_{\om} \Delta y\Delta v + \nabla y \nabla v + yv\, dx$ and it is compactly embedded in $H_{0}^{1}(\om)$ \cite{christof}.
% {\color{red}Provide reference!}.

We will also denote by $F$ and $N$ the Nemytskii operators $y\mapsto F(y)$ and $y\mapsto N(y)$, with $F(y)(x)=f(x,y)$, $N(y)(x)=\mathcal{N}(x,y)$ for $y$ in some $L^{p}$ space. Note that under the assumptions \eqref{assumption_1}--\eqref{assumption_2}, $F$ is a well-defined operator from $L^{\infty}(\om)\to L^{\infty}(\om)$ which is Lipschitz continuous on bounded sets. Morever $N: L^{p}(\om)\to L^{p}(\om)$ is Lipschitz continuous for every fixed $1\le p\le \infty $.

\subsection{Existence and uniqueness}

The next proposition deals with the existence of solutions for the state equation of \eqref{P_intro} as well as continuity properties of the corresponding control-to-state map. As the proof is rather standard, we will only provide a sketch. We also note that here we closely follow \cite[Proposition 2.1]{christof} where $f(x,y)=\mathrm{max}(y,0)$.

\begin{proposition}\label{state_existence}
	Let $f:\om\times \RR\to \RR$ satisfy the assumptions \eqref{assumption_1}--\eqref{assumption_3}, and let $u\in L^{p}(\om)$ for some $p\ge 2$ and $p>\frac{d}{2}$. Then the state equation
	\begin{equation}\label{state}\tag{$E$}
	\left \{
	\begin{aligned}
	-\Delta y + f(\cdot, y)&=u, \;\; \text{ in }\Omega,\\
	y&=0, \;\; \text{ on }\partial \Omega,
	\end{aligned}
	\right. 
	\end{equation}
	admits a unique solution $y\in Y\cap C^{0,a}(\overline{\om})$, for some H\"older exponent $a>0$ depending only on $p, d$ and $\om$. Furthermore, for every $M>0$ there exists a constant $c_{a}>0$ (that depends on $M$) such that
	\begin{equation}\label{Holder_estimate}
	\|y\|_{C^{0,a}(\overline{\om})} \le c_{a}\|u-f(\cdot, 0)\|_{L^{p}(\om)}, \quad \text{for all } \|u\|_{L^{p}(\om)}\le M.
	\end{equation} 
	Finally, the control-to-state map $S:u\mapsto y$ has the following properties:
	\begin{enumerate}
		\item $S: L^{p}(\om)\to H_{0}^{1}(\om)$ is globally Lipschitz and weakly-strongly continuous.
		\item $S: L^{p}(\om)\to Y$ is weakly-weakly continuous. 
		\item $S: L^{p}(\om)\to Y\cap C^{0,a}(\overline{\om})$ is Lipschitz continuous on bounded sets of $L^{p}(\om)$. 
	\end{enumerate}
Whenever the nonlinearity is represented by a ReLU network $\mathcal{N}$ (i.e.\ $f=\mathcal{N}$), then the last Lipschitz continuity is also global and the constant $c_{a}$ does not depend on $M$.	
\end{proposition}
\begin{proof}
	We only provide a general sketch of proof for the sake of completeness, since the arguments are standard. The first step is to consider a family of truncations $f_{k}$ of $f$ where for $k>0$ we define 
	\begin{equation}\label{f_truncation}
	f_{k}(x,y):=
	\begin{cases}
	f(x,-k), & \text{ if } y<-k,\\
	f(x,y), & \text { if } |y|\le k,\\
	f(x,k), & \text{ if } y>k.
	\end{cases}
	\end{equation}
	The associated state equation reads
	\begin{equation}\label{state_k}\tag{$E_{k}$}
	\left \{
	\begin{aligned}
	-\Delta y + f_{k}(\cdot, y)&=u, \;\; \text{ in }\Omega,\\
	y&=0, \;\; \text{ on }\partial \Omega.
	\end{aligned}
	\right. 
	\end{equation}
	An application of the Browder-Minty theorem gives that \eqref{state_k} admits a unique solution $y\in H_{0}^{1}(\om)$. Since $u\in L^{2}(\om)$ and $f_{k}$ is bounded, in particular $F_{k}:L^{2}(\om) \to L^{2}(\om)$ so we have that $\Delta y_{k}\in L^{2}(\om)$ as well, and hence $y_{k}\in Y$. The same theorem gives that the corresponding control-to-state map $S_{k}: L^{p}(\om)\to H_{0}^{1}(\om)$ is Lipschitz continuous (with a Lipschitz constant that does not depend on $k$) and also weakly-strongly continuous. In fact $S_{k}: L^{p}(\om)\to Y$ is also globally Lipschitz by using the fact $F_{k}$ is so as well (but here the Lipschitz constant depends on $k$). Moreover, one can  easily check that the  map from $S_{k}: L^{p}(\om)\to Y$ is also weakly-weakly continuous.
	
	By  adapting standard results of Stampachia's method, e.g. \cite[Theorem 4.5]{Tro10}, and crucially also using that $p>\frac{d}{2}$, one can show that $y_{k}\in L^{\infty}(\om)$ and there exists a constant $c_{\infty}>0$ independent of $u$ such that
	\begin{equation}\label{Linfty_bound}
	\|y_{k}\|_{L^{\infty}(\om)}\le c_{\infty} \|u-f_{k}(\cdot, 0)\|_{L^{p}(\om)}=c_{\infty} \|u-f(\cdot, 0)\|_{L^{p}(\om)}.
	\end{equation}
	By choosing a large enough $k>0$, it is clear that $y_{k}=y$ is the solution of the problem \eqref{state}, which is also unique from the monotonicity of $f$. The Lipschitz continuity of $S:L^{p}(\om)\to H_{0}^{1}(\om)$ as well as its weak-strong continuity follow readily. However, for  the Lipschitz continuity of  $S: L^{p}(\om) \to Y$ one needs to restrict to bounded sets of $L^{p}(\om)$ since the corresponding Lipschitz constant for $F_{k}$ depends on $k$. Note that this is not the case when $f$ is globally Lipschitz, e.g., a ReLU neural network.
	
	Finally, regarding the H\"older regularity and the corresponding estimate \eqref{Holder_estimate}, these follow from an application of \cite[Theorem 8.29]{gilbarg2015elliptic}.
% Regarding the Lipschitz continuity with respect to $C^{0,\alpha}$ seminorm, since the proof is non standard we have put it in Appendix \ref{sec:app.proof}.
Thus, it remains to show the Lipschitz continuity with respect to the $C^{0,a}$ semi-norm. For this purpose, let $u_{i}\in L^{p}(\om)$ and $y_{i}=S(u_{i})$, $i=1,2$, and define
\[\xi:=
\begin{cases}
\frac{F(y_{1})-F(y_{2})}{y_{1}-y_{2}} & \text{on } \{x\in\om: y_{1}(x)\ne y_{2}(x)\},\\
0 						      & \text{on } \{x\in\om: y_{1}(x)= y_{2}(x)\}.
\end{cases}
\]
Since $y_{1}, y_{2}\in C^{0,a}(\overline{\om})$, $\xi\ge 0$ is essentially bounded by a local Lipschitz constant of $F$ depending on $\max \{\|y_{1}\|_{\infty}, \|y_{2}\|_{\infty}\}$. Then, by definition of $\xi$ we have that $y_{1}-y_{2}$ solves the equation
	\begin{equation}\label{diff_y1y2}
	\left \{
	\begin{aligned}
	-\Delta y + \tilde{f}( y)&=u_{1}-u_{2}, \;\; \text{ in }\Omega,\\
	y&=0, \;\quad\quad\quad \text{ on }\partial \Omega,
	\end{aligned}
	\right. 
	\end{equation}
	where $\tilde{f}(y)=\xi y$. Obviously, the function $\tilde{f}$ satisfies the assumptions of this  proposition, and in particular we can apply the estimate \eqref{Holder_estimate} to obtain
	\[\|y_{1}-y_{2}\|_{C^{0,a}(\overline{\om})}\le \tilde{c}_{a}\|u_{1}-u_{2}-\tilde{f}(0)\|_{L^{p}(\om)}=\tilde{c}_{a}\|u_{1}-u_{2}\|_{L^{p}(\om)}.\]
	Note that, also in view of the discussion right after, the constant $\tilde{c}_{a}$ depends on the Lipschitz constant of $\tilde{f}$, that is $\|\xi\|_{L^{\infty}(\om)}$, which is controlled by $\|y_{1}\|_{L^{\infty}(\om)}+ \|y_{2}\|_{L^{\infty}(\om)}$. The latter can be controlled by $\|u_{1}\|_{L^{p}(\om)}, \|u_{2}\|_{L^{p}(\om)}$. Thus, in the end we get the required  Lipschitz continuity on bounded sets of $L^{p}(\om)$. Again for a globally Lipschitz $f$, as it is the case for a ReLU network $\mathcal{N}$, the Lipschitz continuity $S: L^{p}(\om)\to C^{0,a}(\overline{\om})$ is also global.	
\end{proof}

We add a remark regarding the constant $M>0$ in the statement of Proposition \ref{state_existence}. We note that in the case of locally Lipschitz $f$, the constant $c_{a}$ in \eqref{Holder_estimate} will also depend on $\|u\|_{L^{p}(\om)}$. This is because   $c_{a}$ depends on the local Lipschitz constant $L>0$ such that $\|f(\cdot, y)-f(\cdot, 0)\|_{L^{p}(\om)}\le L \|y\|_{L^{p}}$. This constant $L$  depends on $\|y\|_{L^{\infty}(\om)}$ which is controlled by $\|u\|_{L^{p}(\om)}$. More specifically, for $u$ in a bounded subset of $L^{p}(\om)$, one can take a uniform constant $c_{a}$ in \eqref{Holder_estimate}. If $f$ is globally Lipschitz, then $c_{a}$ can be chosen independently of $\|u\|_{L^{p}(\om)}$.

\subsection{Control-to-state map (differentiability)}
We will proceed by stating differentiability properties of the solution operator $S$ of the state equation \eqref{state}, under the assumptions \eqref{assumption_1}--\eqref{assumption_4}.

We start by mentioning that it is easy to check after an application of the dominated convergence theorem that under these assumptions, $F:L^{\infty}(\om)\to L^{p}(\om)$ is Hadamard directionally differentiable  for $1\le p<\infty$. The directional  derivative $F'(y;h)\in L^{p}(\om)$ is given by
\begin{equation}\label{F_dir_deriv}
F'(y;h)(x)=f_{x}'(y(x);h(x)).
\end{equation}
In fact if $y,h\in L^{\infty}(\om)$ and $(h_{n})_{n\in\NN}$ is a bounded sequence in $L^{\infty}(\om)$ with $h_{n}\to h$ in $L^{p}(\om)$, and $t_{n}\to 0^{+}$, then we have
\begin{equation}\label{had_diff_stronger}
\frac{F(y+t_{n}h_{n})-F(y)}{t_{n}} \to F'(y;h) \quad \text{ in } L^{p}(\om).
\end{equation}
In the case of a ReLU neural network,  we also have that $N:L^{p}(\om)\to L^{p}(\om)$ for $1\le p <\infty$ is Hadamard directional differentiable with the directional derivative $N_{x}'(y;h)\in L^{p}(\om)$, given by the analogous formula to \eqref{F_dir_deriv}.
The following proposition also is a simple adaptation of \cite[Theorem 2.2]{christof} to our case.

\begin{proposition}\label{cts_diff}
	The control-to-state map $S: L^{p}(\om)\to Y$ of \eqref{state} is Hadamard directional differentiable, that is, given $u\in L^{p}(\om)$, a direction $h\in L^{p}(\om)$, $(h_{n})_{n\in\NN}\subset L^{p}(\om)$ with $h_{n} \to h$ in $L^{p}(\om)$ and $t_{n}\to 0^{+}$ we have
	\[\frac{S(u+t_{n}h_{n})-S(u)}{t_{n}}\to S'(u;h) \quad \text{ in }Y.\]
	Moreover, $S'(u;h) :=z_{h}\in Y\cap C^{0,a}(\overline{\om})$ and it is the unique solution of 
	
	\begin{equation}\label{adjoint}\tag{$K$}
	\left \{
	\begin{aligned}
	-\Delta z_{h} + F'(y;z_{h})&=h, \;\; \text{ in }\Omega,\\
	z_{h}&=0, \;\; \text{ on }\partial \Omega,
	\end{aligned}
	\right. 
	\end{equation}
	where $y=S(u)$.
\end{proposition}

\begin{proof}
	We denote $u_{n}:=u+t_{n}h_{n}\in L^{p}(\om)$, $y_{n}:=S(u_{n})$ and $z_{n}:=\frac{y_{n}-y}{t_{n}}$ for every $n\in\NN$. From the Lipschitz continuity of $S$, we deduce the existence of some $z_{h}\in Y$ such that $z_{n}\rightharpoonup z_{h}$ (weakly) in $Y$ and $z_{n}\to z_{h}$ (strongly) in $H_{0}^{1}(\om)$ (due to the compact embedding) along a subsequence. Furthermore we have
	\begin{equation}\label{quot}
	-\Delta z_{n} + \frac{F(y_{n})- F(y)}{t_{n}}=h_{n}.
	\end{equation}
	According to \eqref{had_diff_stronger}, the second  term of the left hand side  in \eqref{quot} converges to $F'(y,z_{h})$ 
	strongly in $L^{2}(\om)$, and since $h_{n}\to h$ in $L^{2}(\om)$, we  also get that $-\Delta z_{n}$ converges strongly to $-\Delta z_{h}$. Hence $z_{n}\to z_{h}$ strongly in $Y$. Taking the weak limit in $L^{2}(\om)$ in \eqref{quot}, we get that $z_{h}$ satisfies \eqref{adjoint}. Note now that fixing $y\in L^{\infty}(\om)$, we have that $F'(y;\cdot): L^{p}(\om)\to L^{p}(\om)$ is the Nemytskii  operator induced by $(x,z)\mapsto f_{x}'(y(x);z(x))$. It can be easily checked that the latter map satisfies the assumptions \eqref{assumption_1}--\eqref{assumption_3}. Hence, Theorem \ref{state_existence} applies and thus $z_{h}\in C^{0,a}(\overline{\om})$ as well. Since this solution is unique, the convergence $z_{n}\to z_{h}$ holds along the whole sequence.

\end{proof}

\subsection{Approximation results}

In this section we are concerned with certain approximation results regarding the state equation \eqref{state}. In particular, we consider approximating sequences for the nonlinear nonsmooth function $f$, and we show approximation results for the  sequence of solutions of the corresponding approximating state equations to the solution of the limit problem. We will consider approximations on two levels: The first level of approximation arises from the approximation of $f$ by a sequence of ReLU neural networks $\mathcal{N}_{n}$ in the sense of \eqref{sobolev_approx_box} and can be thought of as the capability of ReLU network-informed PDEs to approximate some ground truth nonsmooth model. The second level of approximation considers the approximation of a ReLU network-informed PDE by (a sequence of) PDEs that correspond to a  smoothing of the network. Concerning the latter we will also briefly discuss some complications which arise when the smoothed network results from simply smoothing the ReLU activation function. We recall that smoothing a general function $f$ is a typical first step in  standard methods that study the optimal control of PDEs that contain nonsmooth components. In this context, one often regularizes these components and subsequently considers the limit behaviour as regularization vanishes; compare, e.g., \cite{barbu1984optimal, christof, mignot_puel}.

\noindent
We start by studying the approximation of a general function $f$ by a sequence of ReLU neural networks.
%\com{(can we relax the following by dropping the $W_{loc}^{1,\infty}$ assumption? We would need some approximation results of Caratheodory functions by continuous functions in a certain sense...)}
\begin{proposition}\label{prop:state_N_n}
	Suppose that $f:\om \times \RR \to \RR$ satisfies \eqref{assumption_1}--\eqref{assumption_4} with the additional assumption that $f\in W_{loc}^{1,\infty}(\RR^{d}\times \RR)$. Moreover, let $u\in L^{p}(\om)$ for some $p\ge 2$ and $p>\frac{d}{2}$. Then, given $K>0$, there exists a sequence $(\mathcal{N}_{n})_{n\in\mathbb{N}}$ of ReLU neural networks $\mathcal{N}_{n}: \RR^{d}\times \RR \to \RR$, $n\in\mathbb{N}$, such that 
	\begin{equation}\label{w1infty_conv}
	\mathcal{N}_{n} \to f, \quad \text{ in } W^{1,\infty}(\om\times (-K,K)).
	\end{equation}
	Furthermore, there exists $K>0$ such that for sufficiently large $n\in \NN$, the approximating learning-informed PDE
	\begin{equation}\label{stateN}\tag{$E_{\mathcal{N}_{n}}$}
	\left \{
	\begin{aligned}
	-\Delta y + \mathcal{N}_{n}(\cdot, y)&=u, \;\; \text{ in }\Omega,\\
	y&=0, \;\; \text{ on }\partial \Omega,
	\end{aligned}
	\right. 
	\end{equation}
	has a unique solution $y_{n}\in Y\cap C^{0,a}(\overline{\om})$ satisfying \eqref{Holder_estimate}, with the corresponding control-to-state-map $S_{n}$ satisfying $(i)$--$(iii)$ of Proposition \ref{state_existence} (with global Lipschitz constants). Furthermore, it holds that 
	\begin{equation}
	y_{n} \to y, \quad \text{ strongly in } Y \text{ and in } C^{0,a}(\overline{\om}),
	\end{equation}
	where $y\in Y\cap C^{0,a}(\overline{\om})$ is the unique solution of \eqref{state}.
\end{proposition}

\begin{proof}
	Given $K>0$, the $W^{1,\infty}(\om\times (-K,K))$ convergence \eqref{w1infty_conv} for some sequence $(\mathcal{N}_{n})_{n\in\NN}$ follows directly from \eqref{sobolev_approx_box}. 
	Note that every $\mathcal{N}_{n}$ satisfies assumptions \eqref{assumption_1}--\eqref{assumption_4} with the possible exception of \eqref{assumption_3}, that is, monotonicity in the second variable. We note however that this monotonicity was only used previously in the application of Browder-Minty's theorem in order to show that the operator $A: H_{0}^{1}(\om)\to H^{-1}(\om)$ with 
	\[\langle A(y), z \rangle_{H^{-1}(\om), H_{0}^{1}(\om)}:= \int_{\om} \nabla y \nabla z\,dx + \int_{\om} f(x,y) z\, dx,\]
	is strongly monotone. We argue that this is the case also for the corresponding operator $A_{n,k}$ (where $f$ has been substituted by the truncation $\mathcal{N}_{n,k}$ of $\mathcal{N}_{n}$ as in \eqref{f_truncation}). Indeed, let $c_{\om}$ to be the Poincar\'e (inequality) constant, set
	\begin{equation}\label{specialK}
	K:=c_{a}(\|u\|_{L^{p}(\om)} + (\|f(\cdot,0)\|_{L^{\infty}(\om)} + 1)|\om|^{1/p}   )>0,
	\end{equation}
	and fix $k>K$.
	Then we have for $y_{1}, y_{2}\in H_{0}^{1}(\om)$
	\begin{equation}\label{Ank_monoto_a}
	\langle A_{n,k}(y_{1})- A_{n,k}(y_{2}), y_{1}-y_{2} \rangle
	\ge \frac{1}{(c_{\om} +1)^{2}} \|y_{1}-y_{2}\|_{H_{0}^{1}(\om)}^{2}
	+ \int_{\om} (\mathcal{N}_{n,k}(x,y_{1}) - \mathcal{N}_{n,k}(x,y_{2})) (y_{1}-y_{2})\,dx.
	\end{equation}
	Due to the fact that $\mathcal{N}_{n,k} \to f_{k}$ in $W^{1,\infty}(\om\times \RR)$, we have that for large enough $n\in\NN$
	\begin{align*}\label{Ank_monoto_b}
	\int_{\om} (\mathcal{N}_{n,k}(x,y_{1}) - \mathcal{N}_{n,k}(x,y_{2})) (y_{1}-y_{2})\,dx
	&=\int_{\om} \big ((\mathcal{N}_{n,k}-f_{k})(x,y_{1})- (\mathcal{N}_{n,k}-f_{k})(x,y_{2})\big )(y_{1}-y_{2})\,dx\\
	&\;\; + \underbrace{\int_{\om} (f_{k}(x,y_{1})-f_{k}(x,y_{2}))(y_{1}-y_{2})\, dx}_{\ge 0}\\
	&\ge -\|\nabla (\mathcal{N}_{n,k}-f_{k})\|_{L^{\infty}(\om\times \RR)}\|y_{1}-y_{2}\|_{L^{2}(\om)}^{2}.
	\end{align*}
	Since the term $\|\nabla (\mathcal{N}_{n,k}-f_{k})\|_{L^{\infty}(\om\times \RR)}$ can be arbitrarily small for large enough $n\in\NN$, we have that the last negative term can be absorbed into $\frac{1}{(c_{\om} +1)^{2}} \|y_{1}-y_{2}\|_{H_{0}^{1}(\om)}^{2}$ in \eqref{Ank_monoto_a} and hence the operator $A_{n,k}$ is strongly monotone for large enough $n\in\NN$. Thus the corresponding approximating truncated problem $(E_{\mathcal{N}_{n,k}})$ has a solution. In order to show that \eqref{stateN} has a solution $y_{n}\in Y\cap C^{0,a}(\overline{\om})$ as well, we argue exactly as in the proof of Proposition \ref{state_existence}, keeping in mind that since $\|\mathcal{N}_{n,k}(\cdot,0)-f(\cdot,0)\|_{L^{\infty}(\om)}=\|\mathcal{N}_{n}(\cdot,0)-f(\cdot,0)\|_{L^{\infty}(\om)}\to 0$ for large enough $k\in\NN$ the $L^{\infty}$ norm of $y_{n_{k}}$ will be bounded by $K$. The estimate \eqref{Holder_estimate} and the corresponding continuity properties of the control-to-state map $S_{n}$ follow analogously. We note that since the Lipschitz constants of $\mathcal{N}_{n}$ are uniformly bounded (in particular their $k$-truncation), the same holds for the Lipschitz constants of the maps $S_{n}: L^{p}(\om)\to H_{0}^{1}(\om)$ and $S_{n}: L^{p}(\om)\to Y\cap C^{0,a}(\overline{\om})$.% \com{(to check!)}. 
	
	Using the convergence of $\mathcal{N}_{n}$ to $f$, it is easy to check  that the sequence $(y_{n})_{n\in\NN}$ is bounded in $Y$ and hence there exists an (unrelabeled) subsequence and $\zeta \in Y$ such that $y_{n}\rightharpoonup \zeta$ weakly in $Y$  and $y_{n}\to \zeta$ strongly in $H_{0}^{1}(\om)$. Note that in view of \eqref{Holder_estimate}, we also have $\|\zeta\|_{L^{\infty}(\om)}\le K$. Observe now  that for every $z\in L^{2}(\om)$
	\begin{align*}
	\left |\int_{\om} \mathcal{N}_{n}(x,y_{n})z\, dx -\int_{\om} f(x,\zeta)z\,dx \right |
	&\le \int_{\om} |\mathcal{N}_{n}(x,y_{n})z - \mathcal{N}_{n}(x,\zeta)z|\,dx +\int_{\om} |\mathcal{N}_{n}(x,\zeta)z-f(x,\zeta)z|\, dx\\
	&\le L\|y_{n}-\zeta\|_{L^{2}(\om)} \|z\|_{L^{2}(\om)}+ \|\mathcal{N}_{n}-f\|_{L^{\infty}(\om\times (-K,K))} \|z\|_{L^{2}(\om)}\to 0
	\end{align*}
	as $n \to \infty$, where we also used the fact that the Lipschitz constants $L_{n}$ of $\mathcal{N}_{n}$ on $\om\times [-K,K]$ are uniformly bounded (say by $L$). Hence $\mathcal{N}_{n}(x,y_{n})\rightharpoonup f(x,\zeta)$ in $L^{2}(\om)$ and thus by taking weak limits in \eqref{stateN}, we have that $\zeta\in Y$ is equal to the  solution $y$ of \eqref{state}. By uniqueness the weak convergence $y_{n}\rightharpoonup y$ (in $Y$) holds for the whole sequence. Note that the estimate
	\begin{equation}\label{laplacian_approx}
	\|\Delta y_{n}-\Delta y\|_{L^{2}(\om)}=\|\mathcal{N}_{n}(\cdot, y_{n})-f(\cdot,y)\|_{L^{2}(\om)}\le L\|y_{n}-y\|_{L^{2}(\om)} + |\om|\|\mathcal{N}_{n}-f\|_{L^{\infty}(\om\times (-K,K))}
	\end{equation}
	yields the strong convergence $y_{n}\to y$ in $Y$, keeping in mind that $\Delta$ induces the norm on $Y$. 
	
	Finally we would like to show that in addition we also have $\|y_{n}-y\|_{C^{0,a}(\overline{\om})}\to 0$ as $n\to \infty$. For that purpose note that $w_{n}:=y_{n}-y$ solves the following PDE:
	\begin{equation}\label{state_diff}
	\left \{
	\begin{aligned}
	-\Delta w_{n} &=f(\cdot,y)-\mathcal{N}_{n}(\cdot,y_{n}), \;\; \text{ in }\Omega,\\
	w_n&=0, \,\;\;\qquad\qquad\qquad\quad \text{ on }\partial \Omega.
	\end{aligned}
	\right. 
	\end{equation}
	Using the same arguments as in Proposition  \ref{state_existence}, we have that 
	\[\|y_{n}-y\|_{C^{0,a}(\overline{\om})}=\|w_{n}\|_{C^{0,a}(\overline{\om})}\le c_{a}\|f(\cdot,y)-\mathcal{N}_{n}(\cdot,y_{n})\|_{L^{p}(\om)},\]
	and thus it suffices to show that the latter is going to zero. We have
	\[\|f(\cdot,y)-\mathcal{N}_{n}(\cdot,y_{n})\|_{L^{p}(\om)}
	\le \|f(\cdot,y)-f(\cdot,y_{n})\|_{L^{p}(\om)} + \|\mathcal{N}_{n}(\cdot,y_{n})-f(\cdot,y_{n})\|_{L^{p}(\om)}.\]
	The first term on the right hand side tends to zero since $y_{n}\to y$ a.e.\ (up to a subsequence) and by applying the dominated convergence theorem, while the second term vanishes since $\|\mathcal{N}_{n}-f\|_{L^{\infty}(\om\times (-K,K))}\to 0$. 
\end{proof}

We will now pass to the second level of approximation, which is more related  to techniques in the optimal control of nonsmooth PDEs, where  an optimality system is derived as a limit of optimality systems corresponding to smooth approximations for the PDE. Here, we will initially consider the learning-informed case and study smoothings of the involved ReLU network $\mathcal{N}$, also providing a word of caution about some complications that can result from such smoothings.  

In particular, here we will consider the following ReLU informed semilinear PDE
\begin{equation}\label{state_N}\tag{$E_{\mathcal{N}}$}
\left \{
\begin{aligned}
-\Delta y + \mathcal{N}(\cdot, y)&=u, \;\; \text{ in }\Omega,\\
y&=0, \;\; \text{ on }\partial \Omega,
\end{aligned}
\right. 
\end{equation}
and its corresponding  smoothed version
\begin{equation}\label{state_N_eps}\tag{$E_{\mathcal{N}_{\epsilon}}$}
\left \{
\begin{aligned}
-\Delta y + \mathcal{N}_{\epsilon}(\cdot, y)&=u, \;\; \text{ in }\Omega,\\
y&=0, \;\; \text{ on }\partial \Omega.
\end{aligned}
\right. 
\end{equation}
A very natural way to produce smooth approximations $\mathcal{N}_{\epsilon}$ of $\mathcal{N}$ is  simply via smoothing its activation function. This type of smoothing, here referred to as  \emph{canonical smoothing} of $\mathcal{N}$, is studied in detail in \cite{DonHinPap22b}. 
%We will assume that $\mathcal{N}$ is monotone increasing in the second variable, that is, all the assumptions \eqref{assumption_1}--\eqref{assumption_4} hold. 
%Of course as we have seen in the previous section $\mathcal{N}_{\epsilon}$ is not monotone increasing necessarily, but as we will see in the next proposition, \eqref{cs_Lp_derivatives} will guarantee uniqueness of solutions for small enough $\epsilon>0$.
%---------SHORTEN THIS AND MOVE THIS DISCUSSION IN SECOND PART---------\\
Note, however, that nonmonotonicity of $\mathcal{N}_{\epsilon}$ may arise after such a smoothing. In fact, as we saw in Proposition \ref{prop:state_N_n}, we were able to prove strong monotonicity for the operator $A_{n,k}$ (for large enough $n$) despite the fact that $\mathcal{N}_{n}$ was not monotone itself. This was done by taking advantage of the $L^{\infty}$ convergence of $\nabla \mathcal{N}_{n,k}$ to $\nabla f_{k}$, the latter being positive almost everywhere. But in the case of a canonical smoothing $\mathcal{N}_{\epsilon}$, the convergence of $\nabla \mathcal{N}_{\epsilon}$ to $\nabla \mathcal{N}$ as $\epsilon\to 0$ can only be guaranteed to hold with respect to the $L^{p}$ norm for every $1\le p<\infty$, see \cite[Remark 2.9]{DonHinPap22b} making the application of the Browder-Minty theorem problematic, see the corresponding section in \cite{DonHinPap22b} .

We stress here that our purpose of using $\mathcal{N}_{\epsilon}$ is to obtain a limiting optimality system as $\epsilon\to 0$ for the optimal control with respect to the learning-informed PDE \eqref{P}.  
Hence, for the following proposition, as well as for the passage to the limit later on, we will assume that $\mathcal{N}_{\epsilon}$ remains monotone with respect to the second variable, a property that can be achieved theoretically (but perhaps not practically in implementations) via convolution; see also Remark \ref{rem:state_f_eps}. We thus set here
\begin{equation}\label{N_eps_mollified}
\mathcal{N}_{\epsilon}(x,y):=\int_{\mathbb{R}} \rho_{\epsilon} (y-z)\mathcal{N}(x,z)\,dz,
\end{equation}
where $(\rho_{\epsilon})_{\epsilon>0}$ is the family of standard  mollifiers \cite{yosidafunctional}. Note that we have $\mathcal{N}_{\epsilon}\to \mathcal{N}$ uniformly as $\epsilon \to 0$, $\|\nabla \mathcal{N}_{\epsilon}\|_{L^{\infty}(\om\times \RR)}$ is uniformly bounded, and each $\mathcal{N}_{\epsilon}$ satisfies the assumptions \eqref{assumption_1}--\eqref{assumption_4}.

Due to the linear growth of $\mathcal{N}_{\epsilon}$ and the $L^{\infty}$ bound for its gradient, we have that the Nemytskii operator $N_{\epsilon}$ with $N_{\epsilon}(y)(x)=\mathcal{N}_{\epsilon}(x,y)$ is continuously Fr\'echet differentiable from $L^{p}(\om)\to L^{q}(\om)$ for every $1\le q<p<\infty$, with derivative $N_{\epsilon}': L^{p}(\om)\to \mathcal{L}(L^{p}(\om),L^{q}(\om)$) defined as
\begin{equation}
(N_{\epsilon}'(y)h)(x)= \nabla_{y}\mathcal{N}_{\epsilon}(x,y(x))  h(x), \quad \text{ for a.e. }x\in\om.
\end{equation}
Consequently, $N_{\epsilon}$ is also continuously Fr\'echet differentiable as a function from $H_{0}^{1}(\om)\to L^{q}(\om)$, for any $1\le q< \frac{2d}{d-2}$ ($d\ge 3$), for any $1\le q<\infty$, ($d=2$), and for any $1\le q\le \infty$ ($d=1$). In particular, $N_{\epsilon}$ is continuously Fr\'echet differentiable from $H_{0}^{1}(\om)\to L^{2}(\om)$ for any dimension $d\ge 1$. 
Recall that $\nabla_{y}\mathcal{N}_{\epsilon}(x,y(x)) \in L^{\infty}(\om \times \RR)$ with its explicit form derived
	\begin{equation}\label{relu_eps_prime}
	\nabla \mathcal{N}_{\epsilon}(x)=W_{L}  \cdot \sigma_{\epsilon}'( N_{\epsilon}^{(L-1)}(x)) \cdot W_{L-1} \cdot \ldots \cdot \sigma_{\epsilon}'(N_{\epsilon}^{(1)}(x))\cdot W_{1}.
	\end{equation}

\begin{proposition}\label{prop:state_N_eps}
	Let $\mathcal{N}:\RR^{d}\times \RR \to \RR$ be a ReLU network that satisfies assumptions \eqref{assumption_1}--\eqref{assumption_4} and let $(\mathcal{N}_{\epsilon})_{\epsilon>0}$ be a  smoothing of $\mathcal{N}$ defined as in \eqref{N_eps_mollified}. Moreover, let $u\in L^{p}(\om)$ for some $p\ge 2$ and $p> \frac{d}{2}$. Then 
	%for small enough $\epsilon>0$, 
	the  smoothed state equation \eqref{state_N_eps} has a unique solution $y_{\epsilon}\in Y\cap C^{0,a}(\overline{\om})$. Furthermore the following hold:
	\begin{enumerate}
		\item The corresponding control-to-state map $S_{\epsilon}$ satisfies the continuity properties  $(i)$--$(iii)$ of Proposition \ref{state_existence} (with global Lipschitz constants independent of $\epsilon>0$).
		\item $S_{\epsilon}: L^{p}(\om)\to Y$ is Fr\'echet differentiable. Its derivative at $u\in L^{p}(\om)$ in direction $h\in L^{p}(\om)$ is given by the unique solution of
		\begin{equation}\label{adjoint_eps}\tag{$K_{\epsilon}$}
		\left \{
		\begin{aligned}
		-\Delta z_{h} + N_{\epsilon}'(y_{\epsilon})z_{h}&=h, \;\; \text{ in }\Omega,\\
		z_{h}&=0, \;\; \text{ on }\partial \Omega,
		\end{aligned}
		\right. 
		\end{equation}
		where $y_{\epsilon}=S_{\epsilon}(u)$.
%		\item There exists a constant $c>0$ such that for every $u\in L^{p}(\om)$ and for every $\epsilon>0$ it holds
%		\begin{equation}\label{S_S_eps}
%		\|S_{\epsilon}(u) - S(u)\|_{Y} \le c \epsilon
%		\end{equation}
%		and furthermore $\|S_{\epsilon}(u)-S(u)\|_{C^{0,\alpha}(\overline{\om})}\to 0$ as $\epsilon\to 0$.
%		%\item \com{(statement about the limit $S_{\epsilon}'$ as $\epsilon\to 0$)}
	\end{enumerate}
\end{proposition}

\begin{proof}
	Since each $\mathcal{N}_{\epsilon}$ satisfies the assumptions \eqref{assumption_1}--\eqref{assumption_4} and has at most linear growth in the second variable, we can directly apply the Browder-Minty theorem, i.e.\ without the truncation argument, and as before show that the smoothed state equation \eqref{state_N_eps} has a unique solution $y_{\epsilon}\in Y\cap C^{0,a}(\overline{\om})$. 
	
	For $(i)$ note that here the continuity properties $(i)$-$(iii)$ of Proposition \ref{state_existence} follow as before. The fact that the Lipschitz constants are global, i.e.\ independent of $\epsilon>0$, follows from the fact that $\|\nabla \mathcal{N}_{\epsilon}\|_{L^{\infty}}$ 
	and hence the Lipschitz constants of $\mathcal{N}_{\epsilon}$ are uniformly bounded in $\epsilon$.
	
	For $(ii)$, note first that since $N_{\epsilon}'(y_{\epsilon})$ is nonnegative, the equation \eqref{adjoint_eps} has a unique solution $z_{h}\in Y\cap C^{0,\alpha}(\overline{\om})$. Then, given the continuous Fr\'echet differentiability of $N_{\epsilon}$, an application of the implicit function theorem gives the  Fr\'echet differentiability of $S_{\epsilon}: L^{p}(\om)\to Y$.
	
%	For $(iii)$, notice first that by subtracting the two equations \eqref{state_N} and \eqref{state_N_eps}, adding and subtracting $N(y_{\epsilon})$,  testing with $y_{\epsilon}-y$, and using the monotonicity of $\mathcal{N}, $ we have for some constant $c>0$ (which will be changing values)
%	\begin{align}
%	\|\nabla (y_{\epsilon}-y)\|_{L^{2}(\om)}\le \int_{\om} (N_{\epsilon}(y_{\epsilon})-N(y_{\epsilon}))(y-y_{\epsilon})\,dx\le c\epsilon,
%	\end{align}
%	Here we used the fact that $(y_{\epsilon})_{\epsilon>0}$ if bounded in $Y$
%	%(since $S_{\epsilon}: L^{p}(\om)\to Y$ is Lipschitz with a global Lipschitz constant),
%	and also \eqref{cs_uniform_values_stronger}. An application of Poincar\'e's inequality gives $\|y_{\epsilon}-y\|_{H_{0}^{1}(\om)}\le c \epsilon$. Then a similar argument to \eqref{laplacian_approx}, gives $\|y_{\epsilon}-y\|_{Y}\le c\epsilon$. In order also to show $\|y_{\epsilon}-y\|_{C^{0,\alpha}(\overline{\om})}\to 0$ again we work like in \eqref{state_diff}. 
	%and we have
	%\begin{align*}
	%\|y_{\epsilon}-y\|_{C^{0,\alpha}(\overline{\om})} 
	%&\le c_{\alpha}\| N_{\epsilon}(y_{\epsilon})- N(y)\|_{L^{p}(\om)}
	%\le \|N_{\epsilon}(y_{\epsilon})- N(y_{\epsilon})\|_{L^{p}(\om)} + \|N(y_{\epsilon})-N(y)\|_{L^{p}(\om)}\\
	%&\le c\epsilon + 
	%\end{align*} 
	
\end{proof}
%\com{(Comment about similar results for a general $f$ vs $f_{\epsilon}$)}

\begin{remark}\label{rem:state_f_eps}
One can get results analogous to Proposition \ref{prop:state_N_eps} for the problem involving the function $f$ under assumptions \eqref{assumption_1}--\eqref{assumption_4}. Indeed, using the smoothing $f_{\epsilon}(x,y):= \int_{\mathbb{R}} \rho_{\epsilon}(y-z)f(x,z)dz$ for almost every $x\in\om$, we have that for almost every $x\in \om$, $f_{\epsilon}(x,\cdot)\in C^{\infty}(\mathbb{R})$, and that it is monotonically increasing. Moreover, $f_{\epsilon}$ is uniformly Lipschitz continuous in $y$ on bounded sets, i.e., for every $M>0$, there exists a constant $L=L(M)$ independent of $\epsilon>0$ such that for every $y_{1},y_{2}\in (-M,M)$
\begin{equation}\label{f_eps_uni_Lip}
|f_{\epsilon}(x,y_{1})-f_{\epsilon}(x,y_{2})|\le L|y_{1}-y_{2}|, \quad \text{for almost every }x\in\om.
\end{equation}
Furthemore, for every $M>0$, there exists a constant $c>0$ such that for all $y\in (-M, M)$
\begin{equation}\label{f_eps_minus_f}
|f_{\epsilon}(x,y)-f(x,y)|\le c\epsilon\quad \text{for almost every }x\in\om.
\end{equation}
The analogous inequalities in $L^{p}$ hold for the superposition operators $F_{\epsilon}$ when restricting $y$ to a $L^{\infty}(\om)$ ball. Then by considering 
\begin{equation}\label{state_f_eps}\tag{$E_{f_{\epsilon}}$}
\left \{
\begin{aligned}
-\Delta y + f_{\epsilon}(\cdot, y)&=u, \;\; \text{ in }\Omega,\\
y&=0, \;\; \text{ on }\partial \Omega,
\end{aligned}
\right. 
\end{equation}
we have that analogous results to the ones of Proposition \ref{prop:state_N_eps} hold with the only difference, that the Lipschitz continuity of $S_{\epsilon}: L^{p}(\om)\to Y\cap C^{0,\alpha}(\overline{\om})$ is local (on bounded sets of $L^{p}(\om)$), but still with Lipschitz constants independent of $\epsilon$. Furthermore, in that case we can only show  G\^ateaux differentiability for  $S_{\epsilon}: L^{p}(\om)\to Y$, with the derivative at $u\in L^{p}(\om)$ in direction $h\in L^{p}(\om)$ given by the solution to \eqref{adjoint_eps} where $N_{\epsilon}'$ is substituted by $F_{\epsilon}'(y)(x):=\partial_{y} f_{\epsilon}(x,y(x))$.
\end{remark}

\section{Existence and approximation results of optimal control}
\label{sec:wellposedness}

We return now to the optimal control problem \eqref{P_intro}.
% and its ReLU network analogue  \eqref{P}. 
 To be consistent in our notation, we denote by  (\hypertarget{PNeps}{$P_{\mathcal{N}_{\epsilon}}$}) the optimal control problem corresponding to the state equation \eqref{state_N_eps}. We also define the reduced problem for \eqref{P_intro} as follows:
\begin{equation}\label{P_red}\tag{$\hat{P}_{f}$}
\min_{u\in \mathcal{C}_{ad}} \mathcal{J}(u):= J(S(u),u):=\frac{1}{2}\|S(u)-g\|_{L^{2}(\om)}^{2} +\frac{\alpha}{2} \|u\|^2_{L^{2}(\om)},
\end{equation}
and analogously we define the reduced objectives $\mathcal{J}_{\mathcal{N}}$ and $\mathcal{J}_{\mathcal{N}_{\epsilon}}$. Further, we consider  $\mathcal{C}_{ad}$ to be a nonempty, bounded, closed, convex subset of $L^{p}(\om)$ with $p\geq 2$ and $p>d/2$.

\begin{proposition}\label{oc_existence}
	The optimal control problem 
	\begin{equation}\label{reduced_prop}
	\min_{u\in \mathcal{C}_{ad}} \mathcal{J}(u)
	\end{equation}
	has a solution.
\end{proposition}

\begin{proof}
	The proof is a standard application of the direct method of calculus of variations. Let $(u_{n})_{n\in \NN}$ be a minimizing sequence for $\eqref{reduced_prop}$ in $L^{p}(\om)$. Since $\mathcal{C}_{ad}$ is bounded in $L^{p}(\om)$, there exists a subsequence $u_{n_{k}}\rightharpoonup u$ for some $u\in L^{p}(\om)$, which also belongs to $\mathcal{C}_{ad}$ by Mazur's theorem. From Proposition \ref{state_existence} $(ii)$, we have that $S(u_{n_{k}})\to S(u)$ in $L^{2}(\om)$. Then the results follows from the lower semicontinuity of $\mathcal{J}$.
\end{proof}
\noindent
Observe that in the case $p=2$, boundedness of $\mathcal{C}_{ad}$ is not necessary to get existence for \eqref{reduced_prop}, due to the term $\frac{\alpha}{2}\|u\|_{L^2(\Omega)}^2$ appearing in the objective.

The following proposition summarizes approximation properties of solutions of the optimal control problem \eqref{P_intro} and approximations (\hypertarget{Pfn}{$P_{f_{n}}$}) with $f_n$ close to $f$.

\begin{proposition}\label{oc_approximation}
	Consider the general optimal control problem \eqref{P_intro} and its corresponding control-to-state map $S: L^{p}(\om)\to Y\cap C^{0,a}(\overline{\om})$. Let  \hyperlink{Pfn}{$(P_{f_{n}})$} be an optimal control problem that corresponds to $f_{n}$ which is regarded to approximate $f$ in a certain sense, and denote by $S_{n}: L^{p}(\om)\to Y\cap C^{0,a}(\overline{\om})$ the corresponding approximating control-to-state map. Suppose also that the corresponding Nemytskii operators satisfy
	\begin{equation}\label{F_Fn}
	\|F_{n}(y)-F(y)\|_{L^{p}(\om)}\le \epsilon_{n}, \quad \text{ for all }y\in L^{\infty}(\om) \text{ with } \|y\|_{L^{\infty}(\om)}\le M 
	\end{equation}
	for some $M>0$ and some $\epsilon_{n} \downarrow 0$ possibly depending on $M$. Then the following hold:
	\begin{enumerate}
		\item There exists $c>0$ such that
		\begin{equation}\label{uniform_ctsm}
		\|S_{n}(u)- S(u)\|_{Y} \le c \epsilon_{n}, \quad \text{for all } u\in \mathcal{C}_{ad},
		\end{equation}
		\item For every sequence of minimizers $(u_{n})_{n\in\NN}$ of \hyperlink{Pfn}{$(P_{f_{n}})$} there exists a minimizer $u$ of  \eqref{P_intro} such that, up to a subsequence 
		\[u_{n}\rightharpoonup u\;  \text{ in }L^{p}(\om),\quad u_{n}\to u\; \text{ in }L^{2}(\om), \quad \text{ and }\quad  S_{n}(u_{n})\to  S(u)\; \text{ in } H_{0}^{1}(\om).\]
		\item 
		%If additionally there exist $a,b\in L^{p}(\om)$ such that $\mathcal{C}_{ad}$ is of the form $\mathcal{C}_{ad}=\{u\in L^{p}(\om):\; a\le u \le b, \text{a.e. in }  \om\}$ then for the above subsequence we also have 
If $u_n \to u\; \text{ in } L^p(\Omega)$, then 		
		%\begin{equation*}
		$S_n(u_n) \to S(u) \; \text{ in }\; Y \cap C^{0,a}(\overline{\Omega})$. 
		%\end{equation*}
	\end{enumerate}
If $\mathcal{C}_{ad}$ is of the form $\mathcal{C}_{ad}=\{u\in L^{p}(\om):\; u_a\le u \le u_b, \text{ a.e. in }  \om\}$ for some $u_a,\, u_b\in L^{p}(\om)$, then for the above subsequence we have $u_n \to u\; \text{ in } L^p(\Omega)$ and $(iii)$ holds.
\end{proposition}

\begin{proof}
	$(i)$ We fix $u\in \mathcal{C}_{ad}$ and write $y_{n}:=S_{n}(u)$, $y:=S(u)$. By subtracting the two state equations, adding and subtracting $F(y_{n})$, testing with $y_{n}-y$ and using the monotonicity of $f$ we have
	\begin{align*}
	\|\nabla y_{n}-\nabla y\|_{L^{2}(\om)}^{2}&= -\int_{\om} (F_{n}(y_{n})-F(y_{n}))(y_{n}-y)\, dx - \int_{\om} (F(y_{n})-F(y))(y_{n}-y)\, dx\\
	&\le \|F_{n}(y_{n})-F(y_{n})\|_{L^{2}(\om)} \|y_{n}-y\|_{L^{2}(\om)}\\
	&\le \|F_{n}(y_{n})-F(y_{n})\|_{L^{2}(\om)} \|y_{n}-y\|_{H_{0}^{1}(\om)}.
	\end{align*}
	Using the Poincar\'e inequality, dividing the inequality above by $ \|y_{n}-y\|_{H_{0}^{1}(\om)}$  and using the fact that 
	$\|y_{n}\|_{L^{\infty}(\om)}\le c_{\infty}\|u-F_{n}(0)\|_{L^{p}(\om)}$ as well as \eqref{F_Fn}, we get the existence of  $c>0$ satisfying the an estimate analogous to \eqref{uniform_ctsm} for the $H_{0}^{1}$ norm. The required estimate for the $Y$-norm is shown by considering 
	\[\|\Delta y_{n}-\Delta y\|_{L^{2}} \le \|F_{n}(y_{n}) - F(y_{n})\|_{L^{2}(\om)} + \|F(y_{n})-F(y)\|_{L^{2}(\om)}, \]
	and taking advantage of \eqref{F_Fn} and the local Lipschitz continuity of $F$.\\
	$(ii)$ The first two convergences follow from a direct application of \cite[Theorem 2.3]{DonHinPap20}. We only need to prove the last one. By setting now $y_{n}:=S_{n}(u_{n})$ and performing a similar estimation as before, we get 
	\begin{align*}
	\|y_{n}-y\|_{H_{0}^{1}(\om)} \le c \left(\|F_{n}(y)-F(y)\|_{L^{p}(\om)} + \|u_{n}-u\|_{H^{-1}(\om)}\right)\le c(\epsilon + \|u_{n}-u\|_{H^{-1}(\om)}).
	\end{align*} 
	We then use the fact that $u_{n}\rightharpoonup u$ in $L^{p}(\om)$ implies $u_{n} \to u$ in $H^{-1}(\om)$.\\
	$(iii)$ Suppose that $u_{n}\to u$ in $L^{p}(\om)$. We need to show $y_{n}\to y$ both in $Y$ and $C^{0,a}(\overline{\om})$. For the former one it suffices to show $\Delta y_{n}\to \Delta y$ in $L^{2}(\om)$ which follows from the estimate
	\[\|\Delta y_{n} - \Delta y\|_{L^{2}(\om)} \le \|F_{n}(y_{n})-F(y_{n})\|_{L^{2}(\om)} + \|F(y_{n})-F(y)\|_{L^{2}(\om)} + \|u_{n}-u\|_{L^{2}(\om)}.\]
	Finally, for the convergence $y_{n}\to y$ in $C^{0,a}(\overline{\om})$, we  work similarly to the proof of Lipschitz continuity of $S: L^{p}(\om)\to C^{0,a}(\overline{\om})$ in Proposition \ref{state_existence}. 
%	For each pair $y_{n}, y$ we can find a positive $\xi\in L^{\infty}(\om)$ such that
%	\[F_{n}(y_{n})-F_{n}(y)=\xi (y_{n}-y),\]
%	Substituting this to the equation
%	\[-\Delta(y_{n}-y)= - (F_{n}(y_{n})- F_{n}(y)) + F(y)- F_{n}(y) + u_{n}-u,\]
%	we obtain
%	\[-\Delta(y_{n}-y) + \xi (y_{n}-y)=F(y)- F_{n}(y) + u_{n}-u\]
%	Using the same arguments as in Appendix \ref{sec:app.proof} we obtain the estimate
%	\[\|y_{n} - y \|_{C^{0,a}(\overline{\om})} \le \tilde{c}_{a} \|F_{n}(y)-F(y)\|_{L^{p}(\om)} + \|u_{n}-u\|_{L^{p}(\om)}.\]

Finally note that if $\mathcal{C}_{ad}$ is of the form $\mathcal{C}_{ad}=\{u\in L^{p}(\om):\; u_a\le u \le u_b, \text{ a.e. in }  \om\}$ for some	 $u_a,\, u_b\in L^{p}(\om)$, then for the same subsequence as in $(ii)$, the convergence $u_{n}\to u$ in $L^{p}(\om)$ follows from an application of the dominated convergence theorem,  as well as $u_{n}\to u$ in $L^{2}(\om)$ which implies almost everywhere convergence up to subsequence.	
\end{proof}

\begin{remark}\label{smoothing_remark}
	A few remarks about Proposition \ref{oc_approximation} are in order. The main requirement for the approximation $f_{n}$ to $f$ is to have a well-defined control-to-state map $S_{n}: L^{p}(\om)\to Y\cap C^{0,a}(\overline{\om})$ that satisfies the continuity properties  $(i)$--$(iii)$ of Proposition \ref{state_existence} as well as the approximation property \eqref{F_Fn}. The former will guarantee existence for the optimal control problem and the latter will imply the approximation results of Proposition \ref{oc_approximation}. We note that these will be satisfied for the approximating sequences $\mathcal{N}_{n}\to f$, $\mathcal{N}_{\epsilon}\to \mathcal{N}$, $f_{\epsilon}\to f$ of Propositions \ref{prop:state_N_n}, \ref{prop:state_N_eps} and Remark \ref{rem:state_f_eps}, respectively.
	
%	We also note that an analogous smoothing results to Proposition \ref{prop:state_N_eps} hold also for the non neural network case. In that case one can consider the problem of Proposition \ref{state_existence} corresponding to a nonsmooth $f$ in relation to a smoothed one corresponding to $f_{\epsilon}$ defined as $f_{\epsilon}(x,y):=\int_{\om} \rho_{\epsilon}(y-z)f(x,z)\,dz$ for almost every $x\in\om$.
\end{remark}

\section{Optimality systems (Stationarity conditions)}
\label{sec:stationary}
In this section we study first-order optimality respectively stationarity systems for the optimal control problem under consideration. For the sake of generality, we study problem  \eqref{P_intro} when the nonlinear function $f$ satisfies Assumptions \eqref{assumption_1}--\eqref{assumption_4}. Note that this setting covers the case $f=\mathcal{N}$. In what follows we specifically focus on the (reduced) minimization problem \eqref{reduced_prop} and are interested in necessary conditions satisfied by its local minimizers. Recall here that $\overline{u}$ is a local minimizer for \eqref{reduced_prop} if there exists an $r>0$ such that for every $u\in L^{p}(\om)$ satisfying $u\in B(\overline{u},r)\cap \mathcal{C}_{ad}\subset L^{p}(\om)$ we have $\mathcal{J}(\overline{u})\le \mathcal{J}(u)$. 

Depending on whether we take a primal or primal-dual view on characterizing stationarity of a feasible control, but also in view of different prerequisites on the problem data, in the sequel we derive different types of first-order conditions for  \eqref{P_intro} respectively \eqref{reduced_prop}.
 
\textbf{B-stationarity}. We start with a primal first-order necessary optimality condition for $\overline{u}$ which relies on Fermat's principle, corresponding to the directional derivative of $\mathcal{J}$ at $\bar u$ being non-negative along any feasible direction. 

To identify these feasible directions, we define the contingent cone of $\mathcal{C}_{ad}$ at $\overline{u}\in \mathcal{C}_{ad}$ as
\begin{equation}\label{TCad}
T_{\mathcal{C}_{ad}}(\overline{u}):=\{h\in L^{p}(\om): \exists\, t_{n}\downarrow 0 \text{ and }h_{n}\to h\in L^{p}(\om), \text{ such that for all }n\in\NN, \; \overline{u}+t_{n}h_{n}\in \mathcal{C}_{ad} \}.
\end{equation}
It can be shown, see for instance \cite[Lemma 6.34]{bonnans2013perturbation}, that if $\mathcal{C}_{ad}$ is of the form 
\begin{equation}\label{Cad_box}
\mathcal{C}_{ad}=\{u\in L^{p}(\om): u_a(x)\le u(x) \le u_b(x), \text{ for almost every }x\in\Omega \}
\end{equation}
with $u_{a},u_{b}\in L^{\infty}(\om)$, $u_{a}<u_{b}$ almost everywhere, then $T_{\mathcal{C}_{ad}}(\overline{u})$ can be characterized by

\begin{equation}\label{TCad_box}
T_{\mathcal{C}_{ad}}(\overline{u})=
\Bigg \{h\in L^{p}(\om):\;\;
\begin{aligned}
h(x)\ge 0,& \;\; \text{almost everywhere in } \{x\in \om: \overline{u}(x)=u_a(x)\}\\
h(x)\le 0,& \;\; \text{almost everywhere in } \{x\in \om: \overline{u}(x)=u_b(x)\}
\end{aligned}
\Bigg\}.
\end{equation}

\begin{proposition}\label{B_stationarity}[B-stationarity]
	Suppose that $\overline{u}$ is a local minimizer for the problem \eqref{reduced_prop} and let $\overline{y}=S(\overline{u})\in C^{0,a}(\overline{\om})\cap Y$ be the associated state. Then the pair $(\overline{u}, \overline{y})$ satisfies the following variational inequality
	\begin{equation}\label{B_stationarity_VI}
	\mathcal{J}'(\overline{u};h)= \langle \overline{y}-g, S'(\overline{u};h)\rangle + \alpha \langle \overline{u},h\rangle \ge 0, \quad \text{for all } h\in T_{\mathcal{C}_{ad}}(\overline{u}).
	\end{equation}
	
\end{proposition}
\begin{proof}
The proof closely follows  \cite[Theorem 1]{Hint_Mord_Suro}.
	Note that the minimization problem \eqref{reduced_prop} can be equivalently written as
	\begin{equation}\label{reduced_prop_equi}
	\min_{u\in L^{p}(\om)} V(u):=\mathcal{J}(u) + \mathcal{X}_{\mathcal{C}_{ad}}(u),
	\end{equation}
	where $\mathcal{X}_{\mathcal{C}_{ad}}$ denotes the indicator function of $\mathcal{C}_{ad}$. Since $\overline{u}$ is a local minimizer of \eqref{reduced_prop_equi}, we have the following Fermat rule, see \cite[Theorem 6.1.9]{aubin2009set}: 
	\begin{equation}\label{Dup_fermat}
	D_{\uparrow} V(\overline{u})(h) \ge 0, \quad \text{for all } h\in L^{p}(\om),
	\end{equation}
	where for a $u\in \mathrm{dom}(V)$, $D_{\uparrow} V(u)(h)$ denotes the contingent epiderivative of $V$ at $u$ in direction $h\in L^{p}(\om)$, given by
	\begin{equation}\label{Dup_def}
	D_{\uparrow} V(u)(h) =\liminf_{\substack{t\downarrow 0 \\ h'\to h}} \frac{V(u+th')- V(u)}{t}.
	\end{equation}
	Hence from \eqref{reduced_prop_equi}, \eqref{Dup_fermat} and \eqref{Dup_def} we have
	\begin{equation}\label{Dup_VI}
	\liminf_{\substack{t\downarrow 0 \\ h'\to h}}  
	\left(
	\frac{\mathcal{J}(\overline{u}+th')-\mathcal{J}(\overline{u})}{t}
	+ \frac{\mathcal{X}_{\mathcal{C}_{ad}}(\overline{u}+th')-\mathcal{X}_{\mathcal{C}_{ad}}(\overline{u})}{t}
	\right)\ge 0 \quad \text{for all } h\in L^{p}(\om).
	\end{equation}
	Note that for all $h\in L^{p}(\om)$ we have
	\begin{equation}\label{Dup_VI_break}
	\begin{aligned}
	\liminf_{\substack{t\downarrow 0 \\ h'\to h}}  
	\left(
	\frac{\mathcal{J}(\overline{u}+th')-\mathcal{J}(\overline{u})}{t}
	+ \frac{\mathcal{X}_{\mathcal{C}_{ad}}(\overline{u}+th')-\mathcal{X}_{\mathcal{C}_{ad}}(\overline{u})}{t}
	\right)&\ge
	\liminf_{\substack{t\downarrow 0 \\ h'\to h}}  
	\frac{\mathcal{J}(\overline{u}+th')-\mathcal{J}(\overline{u})}{t}\\
	&\;\;+
	\liminf_{\substack{t\downarrow 0 \\ h'\to h}}  
	\frac{\mathcal{X}_{\mathcal{C}_{ad}}(\overline{u}+th')-\mathcal{X}_{\mathcal{C}_{ad}}(\overline{u})}{t}.
	\end{aligned}
	\end{equation}
	Furthermore, it holds that
	\begin{equation}\label{Dup_XCad}
	\liminf_{\substack{t\downarrow 0 \\ h'\to h}}  
	\frac{\mathcal{X}_{\mathcal{C}_{ad}}(\overline{u}+th')-\mathcal{X}_{\mathcal{C}_{ad}}(\overline{u})}{t}=
	\liminf_{\substack{t\downarrow 0 \\ h'\to h}}  
	\frac{\mathcal{X}_{\mathcal{C}_{ad}}(\overline{u}+th')}{t}=
	\begin{cases}
	0, & \text{ if }h\in T_{\mathcal{C}_{ad}}(\overline{u})\\
	\infty, & \text{ otherwise. }
	\end{cases}
	\end{equation}
	Thus, given that $\mathcal{J}'(\overline{u};h)<\infty$ for every $h\in L^{p}(\om)$, we have that the expression  \eqref{Dup_VI} is equivalent to the corresponding one, where $h\in T_{\mathcal{C}_{ad}}(\overline{u})$ only. For those $h$ we have
	\begin{equation}
	\liminf_{\substack{t\downarrow 0 \\ h'\to h}}  
	\left(
	\frac{\mathcal{J}(\overline{u}+th')-\mathcal{J}(\overline{u})}{t}
	+ \frac{\mathcal{X}_{\mathcal{C}_{ad}}(\overline{u}+th')-\mathcal{X}_{\mathcal{C}_{ad}}(\overline{u})}{t}
	\right)
	=\liminf_{\substack{t\downarrow 0 \\ h'\to h}}  \frac{\mathcal{J}(\overline{u}+th')-\mathcal{J}(\overline{u})}{t}
	= \mathcal{J}'(\overline{u};h),
	\end{equation}
	and thus \eqref{Dup_VI} is equivalent to
	\begin{equation}\label{Jprime_VI}
	\mathcal{J}'(\overline{u};h)\ge 0 \quad \text{for all } h\in T_{\mathcal{C}_{ad}}(\overline{u}).
	\end{equation}
	By using the chain rule for Hadamard directional differentiability \cite[Proposition 2.47]{bonnans2013perturbation}, we have
	\begin{align*}
	\mathcal{J}'(\overline{u};h)
	&=\partial_{y} J(\overline{y},\overline{u}) S'(\overline{u};h) +\partial_{u} J(\overline{y},\overline{u})h \\
	&=\langle \overline{y}-g, S'(\overline{u};h) \rangle + \alpha \langle \overline{u}, h\rangle,
	\end{align*}
	and the proof is complete.
\end{proof}

We will say that a point $\overline{u}\in \mathcal{C}_{ad}$ with corresponding state $\overline{y}$ is \emph{$B$-stationary} if it satisfies \eqref{B_stationarity_VI}.\\

\textbf{Optimality conditions via regularization}. Next, we are interested in  optimality conditions that are derived  as a limit of the optimality conditions of smooth problems, see the discussion at the end of Section \ref{sec:ReLU_PDE}. Since this approach is not new \cite{barbu1984optimal, christof, mignot_puel}, we will rather sketch the proof of the following proposition and afterwards discuss  some more special cases in more detail.

\begin{proposition}\label{weak_stationarity}[Weak stationarity]
	Suppose that $\overline{u}$ is a local minimizer for the problem \eqref{reduced_prop} with $\overline{y}=S(\overline{u})\in Y\cap C^{0,a}(\overline{\om})$ the associated state. Then there exists a nonnegative $\zeta\in L^{\infty}(\om)$ and an adjoint state $p\in Y$ such that the following hold
\begin{align}
-\Delta p +\zeta p&=\overline{y}-g,\label{weak_station_1}\\
\langle p+\alpha \overline{u}, u-\overline{u} \rangle & \ge 0\quad \text{for all }u\in \mathcal{C}_{ad}. \label{weak_station_2}
\end{align}	
If $\mathcal{C}_{ad}$ is given by \eqref{Cad_box}, then \eqref{weak_station_2} can be equivalently formulated as follows: There exists a multiplier $\mu=\mu_{b}-\mu_{a}$ with $\mu_{a}, \mu_{b}\in L^{2}(\om)$ such that $\overline{u}$ and $p$ satisfy
\begin{align}
p+\alpha \overline{u}+\mu=0, \label{mu1}\\
u_{a}\le \overline{u} \le u_{b}, \;\; \mu_{a}\ge 0, \;\;\mu_{b}\ge 0, \label{mu2}\\
\mu_{a}(u_{a}-\overline{u})=\mu_{b}(\overline{u}-u_{b})=0, \label{mu3}
\end{align}
almost everywhere in $\om$.
\end{proposition}

\begin{proof}
We will state the proof in several steps:

\emph{Step 1:} Consider \eqref{state_f_eps}, the regularized state equation that corresponds to the smoothing of $f$ discussed in Remark \ref{rem:state_f_eps}. Then by slightly generalizing the standard arguments of \cite[Section 4]{christof}, one can show that there exist $\epsilon>0$ and local minimizers of $(\tilde{P}_{f_{\epsilon}})$ such that $u_{\epsilon}\to \overline{u}$ in $L^{p}(\om)$ as $\epsilon\to 0$, where $(\tilde{P}_{f_{\epsilon}})$ is the following modified regularized problem 

\begin{equation}\label{tP_f_eps}\tag{$\tilde{P}_{f_{\epsilon}}$}
\begin{aligned}
&\underset{ (y,u)\in H_{0}^{1}(\om) \times L^{2}(\om)}{\text{minimize }}\quad \tilde{J}(y,u):= \frac{1}{2} \|y-g\|_{L^{2}(\Omega)}^{2} + \frac{\alpha}{2} \|u\|_{L^{2}(\om)}^{2}+\frac{1}{p} \|u-\overline{u}\|_{L^{p}(\om)}^{p},\\
&\text{subject to }\left \{
\begin{aligned}
-\Delta y + F_{\epsilon}(y)&=u, \;\; \text{ in }\Omega,\\
y&=0, \;\; \text{ on }\partial \Omega,
\end{aligned}
\right. \quad \text{ and }\quad  u\in \mathcal{C}_{ad}.
\end{aligned}
\end{equation}

From Proposition \ref{oc_approximation} we have that $y_{\epsilon}=S_{\epsilon}(u_{\epsilon})\to S(\overline{u})=y$ in $Y\cap C^{0,a}(\overline{\om})$ and from the Lipschitz continuity of $S_{\epsilon}: L^{p}(\om)\to Y\cap  C^{0,a}(\overline{\om})$ on bounded sets of $L^{p}(\om)$ and with Lipschitz constant independent of $\epsilon$ we get that $\|y_{\epsilon}\|_{L^{\infty}(\om)}\le c$ uniformly in $\epsilon$ for some constant $c>0$. Combining this with the fact that $F_{\epsilon}'(y)(x)=\partial_{y}f_{\epsilon}(x,y(x))$ is a bounded operator from $L^{\infty}(\om)$ to itself due to the uniform Lipschitz constant of $f_{\epsilon}$ we get that
\begin{equation}
\|F_{\epsilon}'(y_{\epsilon})\|_{L^{\infty}(\om)}\le c,
\end{equation}
with a constant $c>0$ independent of $\epsilon$. Hence, along a subsequence, $F_{\epsilon}'(y_{\epsilon})\overset{\ast}{\rightharpoonup} \zeta$ weakly* in $L^{\infty}(\om)$ for some $\zeta$ with $\|\zeta\|_{L^{\infty}(\om)}\le c$.

\emph{Step 2:} It follows from standard arguments \cite{Tro10}, that for every local minimizer $u_{\epsilon}$ ($\ne\bar u$) of \eqref{tP_f_eps} there exists a unique adjoint state $p_{\epsilon}\in Y$ such that
\begin{align}
-\Delta p_{\epsilon} + F_{\epsilon}'(y_{\epsilon}) p_{\epsilon} & = y_{\epsilon}-g,\label{smooth_Adj}\\
\langle p_{\epsilon} +\alpha u_{\epsilon}, u-u_{\epsilon} \rangle +
\langle |u_{\epsilon}-\overline{u}|^{p-2} (u_{\epsilon}-\overline{u}), u-u_{\epsilon} \rangle&\ge 0 \quad \text{for all }u\in \mathcal{C}_{ad}. \label{smooth_vi}
\end{align}

\emph{Step 3:} We now take the limit as $\epsilon\to 0 $ in \eqref{smooth_Adj}--\eqref{smooth_vi} using the limiting behaviours of $u_{\epsilon}$, $y_{\epsilon}$ and $F_{\epsilon}'(y_{\epsilon})$. Note first that by testing \eqref{smooth_Adj} with $p_{\epsilon}$, employing the Poincar\'e inequality  and using that $y_{\epsilon}\to y$ in $L^{2}(\om)$, we can easily show that $p_{\epsilon}$ is bounded in $Y$ and thus $p_{\epsilon} \rightharpoonup p$ in $Y$ and $p_{\epsilon} \to p$ in $H_{0}^{1}(\om)$ for some $p\in Y$ along a subsequence. Since $F_{\epsilon}'(y_{\epsilon})\overset{\ast}{\rightharpoonup} \zeta$ in $L^{\infty}(\om)$ and $p_{\epsilon} \to p$ in $L^{2}(\om)$ we deduce $F_{\epsilon}'(y_{\epsilon}) p_{\epsilon} \rightharpoonup \zeta p$ in $L^{2}(\om)$. Since $p_{\epsilon} \rightharpoonup p$ in $Y$ we have $\Delta p_{\epsilon} \rightharpoonup \Delta p$ in $L^{2}(\om)$ and since weak convergence in $L^{2}(\om)$ implies strong convergence in $H^{-1}(\om)$, by taking the limit $\epsilon\to 0 $ in \eqref{smooth_Adj} we have that $p$ satisfies \eqref{weak_station_1}. Finally, we take the limits in \eqref{smooth_vi}. We can straightforwardly estimate that $\||u_{\epsilon}-\overline{u}|^{p-2} (u_{\epsilon}-\overline{u})\|_{L^{q}(\om)}^{q}\to 0$ as $\epsilon \to 0$ where $q=p/(p-1)$. Using this  as well as $p_{\epsilon}\to p$ in $L^{2}(\om)$ and $u_{\epsilon} \to \overline{u}$ in $L^{p}(\om)$, we pass to the limit in \eqref{smooth_vi} and obtain \eqref{weak_station_2}.

Finally we refer to \cite[Theorem 2.29]{Tro10} for a proof of equivalence of \eqref{weak_station_2} and the conditions \eqref{mu1}--\eqref{mu3} in the case of a box constraint set.

\end{proof}

\textbf{C-stationarity}. In general we say that a $\overline{u}\in \mathcal{C}_{ad}$ with state $\overline{y}$ is \emph{weakly stationary} if there exist a non-negative $\zeta\in L^{\infty}(\om)$ and an adjoint state $p\in Y$ such that \eqref{weak_station_1}--\eqref{weak_station_2} are satisfied. Proposition \ref{weak_stationarity} states that every local mimimizer of \eqref{reduced_prop} is weakly stationary.
Note, however, that this weak stationarity condition \eqref{weak_station_1}--\eqref{weak_station_2} does not provide any information about the dual variable $\zeta$. In this context, we are in particular interested in conditions that guarantee 
\begin{equation}\label{C-stationary}\tag{$C$}
\zeta(x)\in \partial f(x, \overline{y}(x)) \quad\text{for almost every }x\in\om,
\end{equation}
where the subdifferential operates on the $y$-component of $f$.
This gives rise to the following stationarity notion: A point $\overline{u}\in \mathcal{C}_{ad}$ with state $\overline{y}$ is called \emph{$C$-stationary} if it is weakly stationary and \eqref{C-stationary} holds in addition. Letting $f_{x}:=f(x,\cdot):\mathbb{R}\to \mathbb{R}$, we write $\partial f(x, y)=:\partial f_{x}(y)$ where the latter represents Clarke's generalized gradient of $f_x$ at $y$. We recall here that for a locally Lipschitz $\phi:\mathbb{R} \to \mathbb{R}$, the Clarke generalized gradient at a point $y$ can be characterized as
\begin{equation}\label{clarke_conv}
\partial \phi(y)=\mathrm{conv} \left ( \{   \xi\in \mathbb{R}: \exists\, y_{n} \text{ with $\phi$ differentiable at $y_{n}$, such that } y_{n}\to y \text{ and } \phi'(y_{n})\to \xi \}\right ); 
\end{equation}
see \cite{Clarke} for a general definition.
In order to show $C$-stationarity for a local minimizer $\overline{u}$ of \eqref{reduced_prop}, one needs to guarantee that the weak$^\ast$ limit $\zeta$ of $F_{\epsilon}'(y_{\epsilon})$ of the regularized nonlinearity at $y_{\epsilon}$ is (in a pointwise sense) an element of the generalized gradient $\partial f(x, \overline{y}(x))$. In the recent preprint \cite{Betz} a similar problem is considered where this difficulty is remedied by assuming that the nonlinearity $f$ is either convex or concave locally around $y_{0}\in \mathbb{R}$ whenever $f$ is nondifferentiable at $y_{0}$ and the set $\{x\in \om: \overline{y}(x)=y_{0}\}$ is of positive measure. This, however, excludes a variety of other nonsmooth monotone functions which might, for instance, be convex on an interval left of $y_{0}$ and concave on an interval on the right. In order to include more general cases, here we will make use of alternative condition, namely piecewise continuous differentiability. 

\begin{assumption}[$PC^{1}(\mathbb{R})$]\label{assum_pc1}
We suppose that apart from assumptions \eqref{assumption_1}--\eqref{assumption_4}, for almost every $x\in\om$ it holds that $f(x,\cdot)$ is piecewise continuously differentiable, and we write $f(x,\cdot)\in PC^{1}(\mathbb{R})$. This means that there exist finitely many points $-\infty<y_{1}<\ldots < y_{k}<+\infty$ such that $f(x,\cdot)$ is continuously differentiable on each of the intervals $(-\infty, y_{1}]$, $[y_{k},\infty)$ and $[y_{j-1}, y_{j}]$ for all $j\in \{2, \ldots k\}$, where the derivatives at $y_{j}$ are understood as one-sided derivatives.
\end{assumption}

\begin{remark}\label{relu_NN_pc1}
It follows immediately that every ReLU neural network $\mathcal{N}: \mathbb{R}^{d}\times \mathbb{R}\to \mathbb{R}$ satisfying \eqref{assumption_1}--\eqref{assumption_4}, that is, every ReLU neural network which is restricted to $\om\times \mathbb{R}$ and is monotone increasing in the second variable also satisfies Assumption \ref{assum_pc1},   since for every $x\in \mathbb{R}^{d}$, $\mathcal{N}(x,\cdot)$ is continuous and  piecewise affine. 
\end{remark}

Note that in view of \eqref{clarke_conv}, it is easy to characterize the Clarke generarized gradient of a $PC^{1}$ function $\phi:\mathbb{R}\to \mathbb{R}$. Indeed,  denoting by $\phi'(\overline{y},h)$ the directional derivative at $\overline{y}$ in direction $h$, and $\phi_{\pm}'(\overline{y}):=\pm \phi'(\overline{y},\pm 1)$ the left- and right-sided derivatives, respectively, we have
\begin{equation*}
\phi_{-}'(\overline{y})=\lim_{y\nearrow \overline{y}} \phi'(y) \quad \text{ and }\quad \phi_{+}'(\overline{y})=\lim_{y\searrow \overline{y}} \phi'(y). 
\end{equation*}
Then, defining $\underline{\partial} \phi=\min \{\phi_{-}', \phi_{+}'\}$ and $\overline{\partial} \phi=\max \{\phi_{-}', \phi_{+}'\}$, it can be checked that $\partial \phi=[\underline{\partial}\phi, \overline{\partial}\phi]$.

Before we proceed towards $C$-stationarity, we need the following lemma.
\begin{lemma}\label{PC1_lemma}
Let $\phi\in PC^{1}(\mathbb{R})$ and let $\phi_{n}:=\rho_{n}\ast \phi$ be the standard mollification of $\phi$ with $\rho_{n}:=\rho_{\epsilon_{n}}$, $\epsilon_{n}\searrow 0$. Suppose that $(y_{n})_{n\in\NN}$ is a sequence in $\mathbb{R}$ such that $y_{n}\to \overline{y}$ for some $\overline{y}\in \mathbb{R}$. Then it holds that
\[\limsup_{n\to\infty} \phi_{n}'(y_{n}) \le \overline{\partial} \phi(\overline{y})
\quad \text{ and } \quad
\liminf_{n\to\infty} \phi_{n}'(y_{n}) \ge \underline{\partial} \phi(\overline{y}).
\]
\end{lemma}
\begin{proof}
See Appendix \ref{sec:app.proof}.
\end{proof}

\begin{proposition}\label{prop:Cstationarity}
Let $f$ satisfy Assumption \ref{assum_pc1}. If $\overline{u}$ is a local minimizer for the for the problem \eqref{reduced_prop} with associated state $\overline{y}$, then $\overline{u}$ is $C$-stationary.
\end{proposition}
\begin{proof}
Since $\overline{u}$ is a local minimizer, by Proposition \ref{weak_stationarity} we have that it is weakly stationary, satisfying \eqref{weak_station_1} for a $\zeta$ obtained via $F_{\epsilon}'(y_{\epsilon}) \overset{\ast}{\rightharpoonup}\zeta$ in $L^{\infty}(\om)$. Since for almost every $x\in\om$, $f(x,\cdot)\in PC^{1}(\mathbb{R})$, and $y_{\epsilon}(x)\to \overline{y}(x)$ as $\epsilon\to 0$, we can apply Lemma \ref{PC1_lemma} and get that for almost every $x\in\om$ we have
\begin{equation}\label{app_PC1_lemma}
\limsup_{\epsilon\searrow 0} f_{\epsilon}'(x,y_{\epsilon}(x))\le \overline{\partial} f(x,\overline{y}(x))
\quad \text{ and }\quad
\liminf_{\epsilon\searrow 0} f_{\epsilon}'(x,y_{\epsilon}(x))\ge \underline{\partial} f(x,\overline{y}(x)).
\end{equation}
From the PC$^{1}$-property of $f$ we have that for almost every $x\in\om$  
\[\partial f (x, \overline{y}(x))=[\underline{\partial} f (x, \overline{y}(x)), \overline{\partial} f (x, \overline{y}(x))].\]
In order to prove \eqref{C-stationary} for $\zeta$ we aim for a contradiction and assume that there exists a set $A\subset \om$ of positive measure such that $\zeta(x) < \underline{\partial} f (x,\overline{y}(x))$ for all $x\in A$. Using \eqref{app_PC1_lemma}, Fatou's lemma and the fact that $F_{\epsilon}'(y_{\epsilon}) \overset{\ast}{\rightharpoonup}\zeta$ in $L^{\infty}(\om)$ we have
\begin{equation*}
\int_{\om} \zeta \mathbbm{1}_{A}\, dx
< \int_{\om} \liminf_{\epsilon\searrow 0} f_{\epsilon}'(x,y_{\epsilon}(x)) \mathbbm{1}_{A}\, dx \le \liminf_{\epsilon\searrow 0} \int_{\om }f_{\epsilon}'(x,y_{\epsilon}(x)) \mathbbm{1}_{A}\, dx 
= \int_{\om} \zeta \mathbbm{1}_{A}\, dx,
\end{equation*}
which is a contradiction. Similarly, via contradiction and using a reversed version of Fatou's lemma for nonnegative functions bounded from above, $\zeta(x)>\overline{\partial} f (x,\overline{y}(x))$ cannot hold on a set of positive measure.
\end{proof}

\textbf{Strong stationarity}. Next we examine conditions that, additionally to $C$-stationarity, will guarantee a certain sign condition on the adjoint state $p$. We say that a point $\overline{u}\in \mathcal{C}_{ad}$ with state $\overline{y}$ satisfies the \emph{sign condition} $\sigma_{(A)}$ on a measurable set $A\subset \om$ if it is $C$-stationary and 
\begin{equation}\label{sign_condition}\tag{$\sigma_{(A)}$}
\zeta(x)p(x) \in [f_{+}'(\overline{y}(x))p(x), f_{-}'(\overline{y}(x))p(x)]\quad \text{ for almost every $x\in A$.}
\end{equation}
Then, $\overline{u}$ is called \emph{strongly  stationary} if it satisfies the sign condition $\sigma_{(\om)}$ on the whole of $\om$. Note that strong stationarity always implies $C$-stationarity even if $f$ does not satisfy the $PC^{1}$ condition. This is due to the fact that $[\underline{\partial} f(x,\overline{y}(x)), \overline{\partial} f (x, \overline{y}(x))]\subset \partial f (x, \overline{y}(x))$ which is a consequence of the directional differentiability of $f$. Note also that on the set $\{x\in\om: p(x)=0\}$ one can change $\zeta$ such that \eqref{C-stationary} is satisfied without affecting \eqref{weak_station_1}.
 
Here we examine conditions under which local minimizers of \eqref{reduced_prop} are  strongly stationary. We will also formulate a constraint qualification under which $B$-stationarity is equivalent to strong stationarity. The next proposition states that the latter always implies the former. We note that these results are inspired by the investigations in the preprint \cite{Betz}.

\begin{proposition}\label{strong_implies_B}
Suppose that $\overline{u}$ is strongly stationary. Then $\overline{u}$ is $B$-stationary.
\end{proposition}
\begin{proof}
For $z\in L^{2}(\om)$, we have $z=z^{+}-z^{-}$ with $z^{+}:=\max(0,z)$, $z^{-}=\max(0, -z)$ and we write
$\om_{z^{+}}:=\{x\in \om: z(x)\ge 0\}$, $\om_{z^{-}}:=\{x\in \om: z(x)\le 0\}$. From the positive homogeneity of the directional derivative with respect to the direction and from the sign condition for $\zeta$ and $p$ we have
\begin{align*}
\int_{\om}  F'(\overline{y};z) p\, dx
&=\int_{\om_{z^+}} F'(\overline{y}; z^{+})p\, dx + \int_{\om_{z^-}} F'(\overline{y}; -z^{-})p\, dx\\ 
&=\int_{\om_{z^+}}  f_{+}'(\overline{y})  pz^{+}\, dx + \int_{\om_{z^-}}   f_{-}'(\overline{y})  p(-z^{-})\, dx\\ 
&\le \int_{\om_{z^+}} \zeta p z^{+}\, dx +  \int_{\om_{z^-}} \zeta p (-z^{-})\, dx= \int_{\om} \zeta p z \, dx
\end{align*}
and hence
\begin{equation}\label{strong_implies_B_1}
\langle \zeta z, p \rangle -  \langle F'(\overline{y};z), p  \rangle \ge 0.
\end{equation}
In view of the adjoint equation \eqref{weak_station_1} we have that $p$ satisfies
\begin{equation}\label{strong_implies_B_2}
\langle -\Delta z+\zeta z, p \rangle 
= \langle \nabla z, \nabla p \rangle + \langle\zeta z, p  \rangle
= \langle -\Delta p, z \rangle+ \langle\zeta p, z  \rangle
= \langle \overline{y}-g, z \rangle, \;\; \text{for all }z\in Y. 
\end{equation}
For $h\in L^{p}(\om)$, let $z:= S'(\overline{u}, h)$. Then, testing the characterizing equation of $z$ from Proposition \ref{cts_diff} with $p$ we get
\begin{equation}\label{strong_implies_B_3}
\langle -\Delta z , p \rangle = - \langle F'(\overline{y}; z), p  \rangle + \langle h, p \rangle.
\end{equation}
Furthermore, we also have that 
\begin{equation}\label{strong_implies_B_4}
\langle p,h  \rangle \ge  -\alpha\langle  \overline{u}, h   \rangle, \quad \text{ for all }h \in T_{\mathcal{C}_{ad}}(\overline{u}). 
\end{equation}
Indeed, using  \eqref{weak_station_2}  we have that this is true for every $h\in \mathcal{R}_{\mathcal{C}_{ad}}= \{t(u-\overline{u}): u\in \mathcal{C}_{ad}, t>0\}$, and then \eqref{strong_implies_B_4} follows by using the fact that $T_{\mathcal{C}_{ad}}(\overline{u})=\overline{\mathcal{R}_{\mathcal{C}_{ad}} (\overline{u})}$ and the continuity of $\langle p+\alpha \overline{u}, \cdot \rangle$. We fix $h\in  T_{\mathcal{C}_{ad}}(\overline{u})$ and let $z=S'(\overline{u};h)$. From \eqref{strong_implies_B_2}, \eqref{strong_implies_B_3} as well as \eqref{strong_implies_B_1} and \eqref{strong_implies_B_4} we have 
\begin{equation}
\langle \overline{y}-g, S'(\overline{u};h)   \rangle 
=\langle \overline{y}-g, z   \rangle 
= \langle \zeta z, p \rangle + \langle -\Delta z, p  \rangle 
=  \langle \zeta z, p \rangle - \langle F'(\overline{y}; z), p  \rangle + \langle h, p \rangle\ge -\alpha \langle \overline{u}, h\rangle.
\end{equation}
Since $h\in  T_{\mathcal{C}_{ad}}(\overline{u})$ is arbitrary, the proof is complete.
 \end{proof}

We now introduce a measurability assumption which will be relevant to the constraint qualification which we state afterwards.
\begin{assumption}\label{measu_nondiff}
The set $\om_{f}:= \{ x\in \om:\,f(x,\cdot) \text{ is nondifferentiable at } \overline{y}(x)\}$ is Lebesgue measurable.
\end{assumption}
 Assumption \ref{measu_nondiff} holds for instance when $f$ satisfies \eqref{assumption_1}--\eqref{assumption_4} and additionally is independent of $x$, or if it satisfies Assumption \ref{assum_pc1} and is jointly continuous on $\om\times \mathbb{R}$. 

We now introduce  a constraint qualification for $\overline{u}\in \mathcal{C}_{ad}$ that lead to strong stationarity as a necessary condition in different settings:
\begin{equation}\label{cq}\tag{$CQ$}
T_{\mathcal{C}_{ad}}(\overline{u}) \text{ is dense in }L^{2}(\om).
\end{equation}
Furthermore, if $\mathcal{C}_{ad}$ is defined via box constraints as in \eqref{Cad_box},  and Assumption \ref{measu_nondiff} holds, then we define
\begin{equation}\label{cqf}\tag{$CQ_{f}$}
\om_{f}\cap \overline{\om}_{a,b} \text{ has Lebesgue measure zero,}
\end{equation} 
where 
\[\om_{a,b}:=\{x\in \om:\; \overline{u}(x)=u_{a}(x) \text{ or } \overline{u}(x)=u_{b}(x) \}.\]

We note that if $\mathcal{C}_{ad}$ is given as a box constraint set as in \eqref{Cad_box} then \eqref{cq} is satisfied if and only if $\om_{a,b}$ has Lebesgue measure zero. This can be seen directly from the characterization \eqref{TCad_box} of $T_{\mathcal{C}_{ad}}(\overline{u}) $ for a box constraint set. Thus if $\om_{a,b}$ is closed, then \eqref{cq} implies \eqref{cqf}. Note also that \eqref{cqf} is satisfied independently of the measure of $\om_{a,b}$ in the case that $f$  is differentiable in the second variable.

\begin{proposition}\label{B_C_implies_strong}
Let $\overline{u}$ with corresponding state $\overline{y}$ be a point which is both $B$-stationary and $C$-stationary. Moreover assume that Assumption \ref{measu_nondiff} holds and that $\mathcal{C}_{ad}$ is defined as in \eqref{Cad_box}. Then $\overline{u}$ satisfies the sign condition $\sigma_{(\om \setminus (\om_{f}\cap \overline{\om}_{a,b}))}$. In particular, if $\eqref{cqf}$ is satisfied, then $\overline{u}$ is strongly stationary.
\end{proposition}

\begin{proof}
We fix an arbitrary direction $h\in T_{\mathcal{C}_{ad} (\overline{u})}$ and let $z:=S'(\overline{u};h)\in Y$. Testing the equation that characterizes $z$ with $p$ and using \eqref{mu1} from weak stationarity, as well as \eqref{B_stationarity_VI} from $B$-stationarity and \eqref{strong_implies_B_2} we get
\begin{equation*}
\langle -\Delta z + F'(\overline{y};z), p \rangle
= \langle h,p  \rangle
=- \alpha\langle \overline{u}, h \rangle - \langle \mu,h\rangle
\le \langle \overline{y}-g, z \rangle - \langle \mu,h\rangle
=   \langle -\Delta z +\zeta z, p\rangle  - \langle \mu,h\rangle,
\end{equation*}
and thus
\begin{equation}\label{B_C_implies_strong_1}
0 \le \langle \zeta z  - F'(\overline{y};z), p \rangle  - \langle \mu,h\rangle
=\int_{\om} (\zeta z- F'(\overline{y};z))p\,dx -\int_{\om}\mu h\,dx
=\int_{\om_{f}} (\zeta z- F'(\overline{y};z))p\,dx -\int_{\om_{a,b}}\mu h\,dx,
\end{equation}
where in the last inequality we used the fact that from $C$-stationarity we have $\zeta(x)z(x)- F'(\overline{y}(x); z(x))=0$ for almost every $x\in \om \setminus \om_{f}$ and   $\mu=0$ on $\om \setminus \om_{a,b}$ which can be easily verified. 

According to Lemma \ref{psi_n} we can find a non-negative $\psi_{m}\in C^{\infty}(\RR^{d})$ such that for an arbitrary nonnegative function $v\in C_{c}^{\infty}(\om)$ the function
\[h_{m}:=-\Delta(v\psi_{m})(x)+ F'(\overline{y};v\psi_{m})(x), \]
belongs to $L^{\infty}(\om)$ and vanishes almost everywhere in $\om_{a,b}$, in particular $h_{m}\in T_{\mathcal{C}_{ad}}(\bar u)$ and it also holds $v\psi_{m}=S'(\overline{u};h_{m})$. Thus we can take $z=v\psi_{m}$ in \eqref{B_C_implies_strong_1} and by using the positive homogeneity of the directional derivative, we get
\[0\le \int_{\om_{f}} (\zeta v \psi_{m} -F'(\overline{y};v\psi_{m}) )p\,dx = \int_{\om} \mathbbm{1}_{\om_{f}} (\zeta-f_{+}'(\overline{y})) \psi_{m}p v\, dx. \]
Since $v\in C_{c}^{\infty}(\om)$ is an arbitrary nonnegative function and from the fundamental lemma of calculus of variations, and the fact that $\psi_{m}$ is also nonnegative,  this implies that 
$(\zeta(x)-f_{+}'(\overline{y}(x)))p(x)\ge 0$ for almost every $x\in\om_{f}\cap \mathrm{supp}\psi_{m}$. Using the pointwise convergence $\mathbbm{1}_{\mathrm{supp}\psi_{m}} \to \mathbbm{1}_{\om\setminus \overline{\om}_{a,b}}$   from Lemma \ref{psi_n} we conclude that
\[\zeta(x) p(x)\ge f_{+}'(\overline{y}(x))p(x)\quad \text{ for almost every }x\in \om_{f}\setminus \overline{\om}_{a,b} .\]
By taking $z=-v\psi_{m}$ in \eqref{B_C_implies_strong_1} we get similarly
\[0\le \int_{\om_{f}} (\zeta (-v \psi_{m}) -F'(\overline{y};-v\psi_{m}) )p\,dx = \int_{\om} \mathbbm{1}_{\om_{f}} (f_{-}'(\overline{y})-\zeta) \psi_{m}p v\, dx \]
and using the same argument we end up with
\[\zeta(x) p(x)\le f_{-}'(\overline{y}(x))p(x)\quad \text{ for almost every }x\in \om_{f}\setminus \overline{\om}_{a,b} .\]
Since for every $x\in \om\setminus \om_{f}$ we have $\zeta(x)p(x)=f'(\overline{y}(x))p(x)$ due to the $C$-stationarity, we deduce that $\overline{u}$ satisfies the sign condition $\sigma_{(\om \setminus (\om_{f}\cap \overline{\om}_{a,b}))}$.
\end{proof}

Note that under Assumptions \ref{assum_pc1} and \ref{measu_nondiff}, the conditions of the previous proposition are satisfied for any local minimizer $\overline{u}$ of \eqref{reduced_prop}. In the next proposition we  show that the converse of Proposition \ref{strong_implies_B} is also true provided that the constraint qualification \eqref{cq} holds. The proof here closely follows  \cite[Proposition 4.13]{christof}.

\begin{proposition}\label{B_implies_strong}
Suppose that one of the following is satisfied for  $\overline{u}$:
\begin{enumerate}
\item The constraint qualification \eqref{cq} holds.
\item The set $\om_{f}$ has zero Lebesgue measure and  $\mathcal{C}_{ad}$ is given as a box constraint set.
\end{enumerate}
Then we have that if $\overline{u}$ is $B$-stationary, it is also strongly stationary.
\end{proposition}
\begin{proof}
We first assume that $\overline{u}$ satisfies \eqref{cq} and it is also $B$-stationary. We set $p:=-\alpha \overline{u}\in L^{2}(\om)$ as a candidate for the adjoint state and it clearly follows that \eqref{weak_station_2} is satisfied. Let $z\in Y$ such that there exists $h\in T_{\mathcal{C}_{ad}}(\overline{u})$ with $z=S'(\overline{u};h)$. Testing the equation that $z$ satisfies with $p$ and using the $B$-stationarity condition yields
\[\langle -\Delta z, p \rangle + \langle F'(\overline{y};z),p \rangle= \langle p, h \rangle=-\alpha \langle \overline{u}, h \rangle \le \langle \overline{y}-g, z\rangle\;\; \Rightarrow \;\; \langle -\Delta z, p \rangle - \langle \overline{y}-g, z\rangle\le - \langle F'(\overline{y};z),p \rangle.\]
We can now interpret $\xi:=\Delta p+ (\overline{y}-g)\in Y^{\ast}$ and estimate
\begin{align}\label{estimate_xi}
\langle \xi, -z \rangle_{Y}
= - \langle \Delta z, p  \rangle - \langle \overline{y}-g, z \rangle
\le - \langle F'(\overline{y};z),p \rangle 
\le \|F'(\overline{y};z)\|_{L^{2}(\om)} \|p\|_{L^{2}(\om)}\le c \|z\|_{L^{2}(\om)} \|p\|_{L^{2}(\om)}
\end{align}
where we used that positive homogeneity of the directional derivative and the fact that $F'(\overline{y};\pm 1)\in L^{\infty}(\om)$. Note that \eqref{estimate_xi} is valid for all $z\in \{S'(\overline{u};h)\in Y:\; h\in T_{\mathcal{C}_{ad}}(\overline{u})\}$. We claim that the latter set is dense in $Y$. Indeed this follows from the fact that the operator $ h\mapsto S'(\overline{u};h)$ derived from equation \eqref{adjoint} is surjective and Lipschitz continuous $L^{p}(\om)\to Y$ from the density of $T_{\mathcal{C}_{ad}}(\overline{u})$ in $L^{2}(\om)$. Using this density we deduce  from \eqref{estimate_xi}
\[\langle \xi, z \rangle_{Y} \le \tilde{c} \|z\|_{L^{2}(\om)},\quad \text{for all }z\in Y.\]
From the Hahn-Banach theorem we can extend $\xi$ from $Y^{\ast}$ to $L^{2}(\om)^{\ast}\simeq L^{2}(\om)$. Denoting its Riesz representation by $\xi $ again we have
\[(\xi,z)_{L^{2}(\om)} = \langle \xi, z \rangle\le -\langle F'(\overline{y};-z), p \rangle, \quad \text{for all }z\in Y.\]
By density of $Y$ in $L^{2}(\om)$ and continuity of $F'(\overline{y};\cdot)$ on $L^{2}(\om)$, we can extend the previous estimate to 
\[\langle \xi, z \rangle\le -\langle F'(\overline{y};-z), p \rangle, \quad \text{for all }z\in L^{2}(\om).\]
For arbitrary $z\in C_{c}^{\infty}(\om)$ with $z\ge0$ we have
\begin{align*}
0&\le \int_{\om} - F'(\overline{y};-z)p-\xi z\, dx = \int_{\om} (f_{-}'(\overline{y})p-\xi)z\, dx,\\
0 &\le \int_{\om} -\xi(-z) - F'(\overline{y};z)p=  \int_{\om} (\xi -f_{+}'(\overline{y})p)z\, dx,
\end{align*}
and from the fundamental lemma of calculus of variations we get
\begin{equation}\label{almost_strong}
f_{+}'(\overline{y}(x))p(x)\le \xi(x) \le  f_{-}'(\overline{y}(x))p(x), \quad \text{ for almost every }x\in\om.
\end{equation}
By defining
\begin{equation}\label{def_zeta}
\zeta(x):=
\begin{cases}
\frac{\xi(x)}{p(x)}, & \text{almost everywhere in } \{p(x)\ne 0\},\\
f_{-}'(\overline{y}(x)), & \text{almost everywhere in } \{p(x)= 0\},
\end{cases}
\end{equation}
we have that $\zeta(x)\in [\underline{\partial}f (\overline{y}(x)), \overline{\partial}f (\overline{y}(x))]$ and thus $\zeta(x)\in \partial f (\overline{y}(x))$ almost everywhere in $\om$. Moreover $\zeta\in L^{\infty}(\om)$, $\zeta\ge 0$ and the sign condition  for $\zeta p$ holds almost everywhere in $\om$. Since $\xi=\zeta p\in L^{2}(\om)$, in view of $\xi=\Delta p + (\overline{y}-g)$ we have
\[-\Delta p + \zeta p=(\overline{y}-g),\]
implying $-\Delta p\in L^{2}(\om)$ and thus $p\in Y$. This concludes the strong stationarity of $\overline{u}$.

Now if we assume that $\overline{u}$ is $B$-stationary and $\om_{f}$ has zero Lebegue measure, then it can be shown, see \cite[Lemma 1.12]{PDEcon_Hinze}, that there exists $\mu=\mu_{b}-\mu_{a}$, with $\mu_{a}, \mu_{b}\in L^{2}(\om)$ such that 

\begin{equation}\label{eq:multi} 
\begin{aligned}
&\mathcal{J}'(\overline{u};h)+\langle \mu,h\rangle= \langle \overline{y}-g, S'(\overline{u};h) \rangle + \alpha\langle \overline{u},h\rangle +\langle \mu ,h\rangle= 0, \quad \text{ for all }  h\in L^{p}(\om),\\
&u_{a}\le \overline{u} \le u_{b}, \;\; \mu_{a}\ge 0, \;\;\mu_{b}\ge 0 \\
&\mu_{a}(u_{a}-\overline{u})=\mu_{b}(\overline{u}-u_{b})=0 
\end{aligned}
\end{equation}
almost everywhere in $\om$. In that case we can  set $p:=-\alpha\overline{u} + \mu \in L^{2}(\om)$ and we proceed as before by using the equality $\{S'(\overline{u};h)\in Y:\; h\in L^{p}(\om)\}=Y$. 

\end{proof}

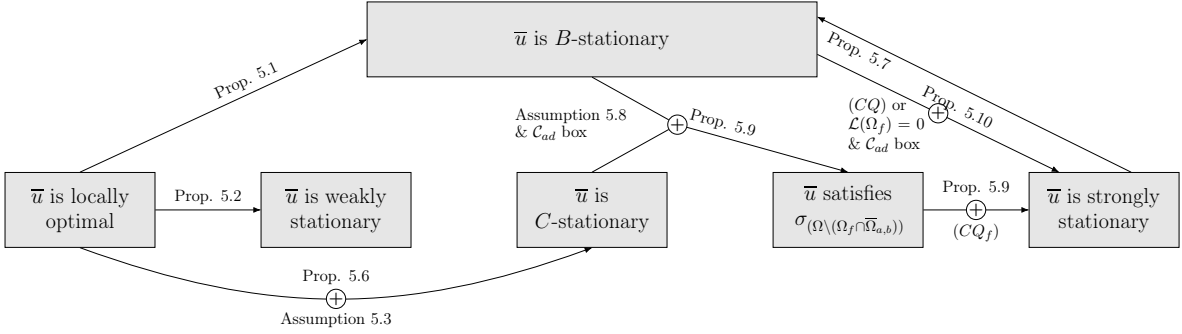
\begin{figure}[h!]
\begin{minipage}{\textwidth}
\centering
\resizebox{0.95\textwidth}{!}{
\begin{tikzpicture}[
node distance= 9mm and 0mm,
rectlong/.style = {draw=black, fill=gray!20, font=\Large,
               minimum width=300pt, minimum height = 50pt},
rect/.style = {draw=black, fill=gray!20, font=\Large,
               minimum width=100pt, minimum height = 50pt},               
every edge/.style = {draw, -Latex},
every edge quotes/.style = {sloped, auto=right}
                ]

\node[rectlong] at (0, 0)  (bb) {$\overline{u}$ is $B$-stationary};
\node[rect, text width=2.5cm, align=center] at (-12, -4)  (a) {$\overline{u}$ is locally optimal};
\node[rect, text width=2.5cm, align=center] at (-6, -4)  (b) {$\overline{u}$ is weakly stationary};
\node[rect, text width=2.8cm, align=center] at (0, -4)  (c) {$\overline{u}$ is \\ $C$-stationary};
\node[rect, text width=2.5cm, align=center] at (6, -4)  (d) {$\overline{u}$ satisfies  $\sigma_{(\Omega\setminus (\Omega_{f}\cap \overline{\Omega}_{a,b}))}$};
\node[rect, text width=2.8cm, align=center] at (12, -4)  (e) {$\overline{u}$ is strongly stationary};
\node[circle, draw=black, fill=white, inner sep=0pt, minimum size=2pt] at (2,-2) (cc) {$\boldsymbol{+}$} ;
\node[text width=3.2cm] at (-0.2,-1.8) {Assumption 5.8}; 
\node[text width=3.2cm] at (-0.2,-2.2) {\& $\mathcal{C}_{ad}$ box}; 
\path   (a.north) edge[pos=0.6, "Prop. 5.1"  swap] (bb.west)
        (a.east) edge["Prop. 5.2" swap] (b.west)
        ;
\path[->]          (a.south)  edge   [bend right=20]   node[above=0.2] {Prop. 5.6} node[below=0.2]{Assumption 5.3} node[circle, draw=black, fill=white, inner sep=0pt, minimum size=2pt] {$\boldsymbol{+}$} (c.south);       
\draw (bb.south)--(cc);
\draw (c.north)--(cc);
\path   (cc) edge[pos=0.2, "Prop. 5.9"  swap ] (d.north);
\path[->]          (d.east)  edge   [bend right=0]   node[above=0.2] {Prop. 5.9} node[below=0.2]{($CQ_{f}$)} node[circle, draw=black, fill=white, inner sep=0pt, minimum size=2pt] {$\boldsymbol{+}$} (e.west);  
\path ([xshift=18pt]e.north) edge[pos=0.85, "Prop. 5.7" ] ([yshift=15pt]bb.east); 
\path[->]          ([yshift=-10pt]bb.east)  edge   [bend right=0,  "Prop. 5.10"{xshift=10pt, swap}]   node[circle, draw=black, fill=white, inner sep=0pt, minimum size=2pt] {$\boldsymbol{+}$} ([xshift=-30pt]e.north);           
\node[text width=2cm] at (7.0,-2) {($CQ$) or $\mathcal{L}(\Omega_{f})=0$ \& $\mathcal{C}_{ad}$ box}; 
\end{tikzpicture}
}
 \end{minipage}
 \caption{Summary of relations between the different stationarity conditions. Arrows represent direct implications which only require the assumptions they start from. Additional assumptions are marked with the ``$\boldsymbol{+}$'' symbol.}
 \label{fig:diagram}
\end{figure}

The different stationarity conditions and their relations are summarized in Figure \ref{fig:diagram}.

\section{Conclusion}\label{sec:conclusion}

In this paper, we studied the optimal control of a family of nonsmooth neural network informed PDEs which approximate some semilinear PDE with unknown nonlinearity. We focused on theoretical investigations, such as well-posedness of the learning-informed PDEs and the optimal control problem, approximation properties, and established several first-order conditions based on purely primal concepts (B-stationarity) and primal-dual systems (e.g.\ C-stationarity and strong stationarity). Particular attention was paid to ReLU networks due to its wide application in practical learning. However, many of the  results presented here fit to optimal control of more general nonsmooth semilinear PDEs.
Different optimality/stationarity conditions require specific assumptions, and their relationships are clarified in Figure \ref{fig:diagram}. We showed the equivalence between strong stationarity and B-stationarity,  under some rather restrictive constraint qualifications. Whether this equivalence persists under milder constraint qualifications remains an open question.

The understanding of optimality conditions for the nonsmooth PDE constrained optimization problem is not only an important theoretical question, but also guides the development of robust nonsmooth numerical algorithms, which is  the focus of the companion paper \cite{DonHinPap22b}.

\subsection*{Acknowledgment}

This work is supported by the Deutsche Forschungsgemeinschaft (DFG, German Research Foundation) under Germany's Excellence Strategy -- The Berlin Mathematics Research Center MATH+ (EXC-2046/1, project ID: 390685689).
The work of GD is also supported by an NSFC grant, No. 12001194.
The work of MH is partially supported by the DFG SPP 1962, project-145r.
The work of KV was supported by the WIAS Female Master Students Program.

\appendix

\section{}\label{sec:app.proof}
\begin{proof}[Proof of Lemma \ref{PC1_lemma}]
If $\phi$ is differentiable $\overline{y}$, then $\underline{\partial} \phi (\overline{y})=\overline{\partial} \phi (\overline{y})=\phi'(\overline{y})$. Since $\phi\in PC^{1}(\mathbb{R})$, $\phi'$ is continuous at a compact interval containing $\overline{y}$. Hence $\phi_{n}'=\rho_{n}\ast \phi' \to \phi'$ uniformly in that interval and thus we have $\lim_{n\to \infty} \phi_{n}'(y_{n})=\phi'(\overline{y})$. 

Suppose now that $\overline{y}$ is a nondifferentiable point of $\phi$. Since  $\phi\in PC^{1}(\mathbb{R})$, there exists an $M>0$ such that $\phi'\in C([\overline{y}-M, \overline{y}])$ and $\phi'\in C([\overline{y}, \overline{y}+M])$. From uniform continuity of $\phi'$ on those intervals we have that for every $\epsilon>0$, there exists a $<\delta\le M$ such that 
\begin{equation}\label{unif_conti}
0<\overline{y}-y<\delta \Rightarrow |\phi'(y)-\phi_{-}'(\overline{y})|<\epsilon \quad \text{ and }\quad
0<y-\overline{y}<\delta \Rightarrow |\phi'(y)-\phi_{+}'(\overline{y})|<\epsilon.
\end{equation}
Since $y_{n} \to \overline{y}$ and $\epsilon_{n}\searrow 0$, it follows that for all $\epsilon>0$ we can pick $N\in \mathbb{N}$ such that for every $n>N$ we have that for all $z\in (-\epsilon_{n}, \epsilon_{n})$
\begin{equation}\label{unif_conti2}
y_{n}-z<\overline{y} \Rightarrow |\phi'(y_{n}-z)-\phi_{-}'(\overline{y})|<\epsilon \quad \text{ and }\quad
y_{n}-z>\overline{y} \Rightarrow |\phi'(y_{n}-z)-\phi_{+}'(\overline{y})|<\epsilon.
\end{equation}
Note also that for every $n$, the set $\{z\in (-\epsilon_{n}, \epsilon_{n}):\, y_{n}-z=\overline{y}\}$ contains at most one point and hence it is a Lebesgue nullset.  In order to prove that $\limsup_{n\to\infty} \phi_{n}'(y_{n}) \le \overline{\partial} \phi(\overline{y})$, we show that for every $\epsilon>0$ we can find $N$ large enough such that for all $n>N$, we have $\phi_{n}'(y_{n})<\overline{\partial} \phi (\overline{y})+\epsilon$. Choosing $N$ as in \eqref{unif_conti2} and using $-\overline{\partial}\phi (\overline{y})\le -\phi'_{\pm}(\overline{y})$ we have
%\begin{align*}
%&\phi_{n}'(y_{n})-\overline{\partial} \phi (\overline{y})=\\
%&=\int_{[-\epsilon_{n}, \epsilon_{n}]}\rho_{n}(z) \left ( \phi'(y_{n}-z) - \overline{\partial} \phi (\overline{y})  \right ) \mathbbm{1}_{\{y_{n}-z<\overline{y}\}}(z)\,dz+\int_{[-\epsilon_{n}, \epsilon_{n}]}\rho_{n}(z) \left ( \phi'(y_{n}-z) - \overline{\partial} \phi (\overline{y})  \right ) \mathbbm{1}_{\{y_{n}-z>\overline{y}\}}(z)\,dz\\
%&\le \int_{[-\epsilon_{n}, \epsilon_{n}]}\rho_{n}(z) \left ( \phi'(y_{n}-z) - \phi_{-}'(\overline{y})  \right ) \mathbbm{1}_{\{y_{n}-z<\overline{y}\}}(z)\,dz
%+\int_{[-\epsilon_{n}, \epsilon_{n}]}\rho_{n}(z) \left ( \phi'(y_{n}-z) - \phi_{+}'(\overline{y})  \right ) \mathbbm{1}_{\{y_{n}-z>\overline{y}\}}(z)\,dz\\
%&<\epsilon \left ( \int_{[-\epsilon_{n}, \epsilon_{n}]}\rho_{n}(z)\mathbbm{1}_{\{y_{n}-z<\overline{y}\}}(z)\,dz
%\int_{[-\epsilon_{n}, \epsilon_{n}]}\rho_{n}(z)\mathbbm{1}_{\{y_{n}-z>\overline{y}\}}(z)\,dz   \right )=\epsilon.
%\end{align*}

\begin{align*}
&\phi_{n}'(y_{n})-\overline{\partial} \phi (\overline{y})=\\
&=\int_{[-\epsilon_{n}, \epsilon_{n}]}\rho_{n}(z) \left ( \phi'(y_{n}-z) - \overline{\partial} \phi (\overline{y})  \right ) \mathbbm{1}_{\{y_{n}-z<\overline{y}\}}(z) + \rho_{n}(z) \left ( \phi'(y_{n}-z) - \overline{\partial} \phi (\overline{y})  \right ) \mathbbm{1}_{\{y_{n}-z>\overline{y}\}}(z)\,dz\\
&\le \int_{[-\epsilon_{n}, \epsilon_{n}]}\rho_{n}(z) \left ( \phi'(y_{n}-z) - \phi_{-}'(\overline{y})  \right ) \mathbbm{1}_{\{y_{n}-z<\overline{y}\}}(z)
+ \rho_{n}(z) \left ( \phi'(y_{n}-z) - \phi_{+}'(\overline{y})  \right ) \mathbbm{1}_{\{y_{n}-z>\overline{y}\}}(z)\,dz\\
&<\epsilon \left ( \int_{[-\epsilon_{n}, \epsilon_{n}]}\rho_{n}(z)\mathbbm{1}_{\{y_{n}-z<\overline{y}\}}(z)\,dz
\int_{[-\epsilon_{n}, \epsilon_{n}]}\rho_{n}(z)\mathbbm{1}_{\{y_{n}-z>\overline{y}\}}(z)\,dz   \right )=\epsilon.
\end{align*}
The inequality $\liminf_{n\to\infty} \phi_{n}'(y_{n}) \ge \underline{\partial} \phi(\overline{y})$ is proven analogously.

\end{proof}

The following lemma and its proof are slight adaptations from \cite[Lemma 3.7]{Betz} which we include here for completeness.
\begin{lemma}\label{psi_n}
For an arbitrary $\mathcal{A}\subset \om$ there exists a sequence $(\psi_{m})_{m\in\NN}\subset C^{\infty}(\RR^{d})$ with $\psi_{m}\ge 0$ such that for every $m\in \NN$ and every $v\in  C^{\infty}(\RR^{d})$
\begin{equation}\label{v_psi_n}
-\Delta(v\psi_{m})(x) + F'(\overline{y};v\psi_{m})(x)=0,\quad \text{ for almost every } x\in \mathcal{A}, 
\end{equation}
and 
\begin{equation}\label{char_conv}
\mathbbm{1}_{\mathrm{supp}\psi_{m}} (x) \to \mathbbm{1}_{\om\setminus \overline{\mathcal{A}}} (x),\quad \text{ for every }x\in\om.
\end{equation}
\end{lemma}
\begin{proof}
For $m\in\NN$, we set $\mathcal{A}_{m}:= \{x\in \mathbb{R}^{d}:\; \mathrm{dist}(\mathcal{A}, x)<\frac{1}{m}\}$, which are open sets, satisfying $\mathcal{A}\subset \mathcal{A}_{m+1}\subset \mathcal{A}_{m}$ for all $m$
and also $\bigcap_{m\in \NN} \mathcal{A}_{m}=\overline{\mathcal{A}}$. According to a classical result by Whitney,  see for instance \cite[Theorem 2.1]{variable_mollifiers}, for every $m\in \NN$ we can find $\psi_{m}\in C^{\infty}(\RR^{d})$ with $\psi_{m}>0$ in $\om\setminus \overline{\mathcal{A}_{m}}$ and $\psi_{m}=0$ in $\RR^{d} \setminus (\om\setminus \overline{\mathcal{A}_{m}})$. In particular, $\psi_{m}=0$ on the open set $\om\cap \mathcal{A}_{m}$ and \eqref{v_psi_n} is satisfied for almost every $x\in \om\cap \mathcal{A}_{m}$ and hence for almost every $x\in \mathcal{A}$. In order to show \eqref{char_conv}, we note that $\mathrm{supp} \psi_{m}= \overline{(\om\setminus \overline{\mathcal{A}_{m}})}=\overline{\om}\setminus \mathcal{A}_{m}$ and $\mathrm{supp} \psi_{m} \subset \mathrm{supp} \psi_{m+1}$ for all $m\in\NN$. Thus we have
\[\bigcup_{m\in\NN} \mathrm{supp}\psi_{m} = \bigcup_{m\in\NN} (\overline{\om}\setminus \mathcal{A}_{m})= \overline{\om} \setminus \left (\bigcap_{m\in\NN} \mathcal{A}_{m}\right )= \overline{\om} \setminus \overline{\mathcal{A}},\]
which implies $\mathbbm{1}_{\mathrm{supp}\psi_{m}}  \to \mathbbm{1}_{\overline{\om}\setminus \overline{\mathcal{A}}} $ pointwise in $\RR^{d}$, and hence \eqref{char_conv} follows.
\end{proof}

\bibliographystyle{plain}
\bibliography{kostasbib}
\end{document}